\documentclass{amsart}

\usepackage{tikz}
\usetikzlibrary{plotmarks}

\usepackage{verbatim}

\usepackage{pgfplots}
\usepgfplotslibrary{fillbetween}

\usepackage[normalem]{ulem}

\newtheorem{theorem}{Theorem}
\newtheorem{lemma}{Lemma}
\newtheorem{prop}{Proposition}
\newtheorem{exam}{Example} 
\newtheorem{corollary}{Corollary}

\newcommand{\ent}{\mathop{\mathrm{ent}}}

\begin{document}

\title{topological entropy on closed sets in $[0,1]^2$}


\author{ Goran Erceg}
\address{Faculty of Science,
University of Split, Rudjera Bo\v{s}kovi\'{c}a 33, 21000 Split, Croatia}
\email{gorerc@pmfst.hr}

\author{Judy Kennedy}
\address{Department of Mathematics,
PO Box 10047,
Lamar University, Beaumont TX 77710,
USA}
\email{kennedy9905@gmail.com }

\today


\maketitle

\begin{abstract}
We generalize the definition of topological entropy due to Adler, Konheim, and McAndrew \cite{AKM} to set-valued functions from a closed subset $A$ of the interval to closed subsets of the interval. We view these set-valued functions, via their graphs, as closed subsets of $[0,1]^{2}$. We show that many of the topological entropy properties of continuous functions of a compact topological space to itself hold in our new setting, but not all. We also compute the topological entropy of some examples, relate the entropy to other dynamical and topological properties of the examples, and we give an example of a closed subset $G$ of $[0,1]^2$ that has $0$ entropy but $G\cup \{ (p,q) \},$ where $(p,q) \in [0,1]^2\setminus G,$ has infinite entropy.

\end{abstract}

\vphantom{}

\section{Introduction}

Generalized inverse limits, or inverse limits with set-valued functions, a subject studied only since 2003 with its introduction by Bill Mahavier, provides an entirely new way to study multi-valued functions, a way that does not lose information under iteration. But it is increasingly apparent that they also offer a rich source of new examples of dynamical systems and continua. In fact, they offer a sort of lab in which one can make mathematical experiments - and then have a real chance, with some effort, of understanding (via some sort of coding) deeply the resulting topology and dynamics of the example, and of how the topology and dynamics are interacting.

\vphantom{}

In this paper we generalize the idea of topological entropy to closed subsets of $[0,1]^{2}$, and later to closed subsets of $[0,1]^n$, for $n$ a positive integer greater than $1$.  We reduce the problem of computing topological entropy in our context to one of counting the ``boxes'' (elements of our grid covers) certain sets generated by our closed subset of $[0,1]^{n}$ intersect. We also relate the topological entropy of the examples we give to the topology and dynamics of the examples.

\vphantom{}

James Kelly and Tim Tennant \cite{KT} have recently studied topological entropy of set-valued functions using Bowen's definition. Our focus is different, as we use the original definition using open covers due to Adler, Konheim and McAndrew [AKM]. There is some overlap of results, and we point these out as they occur. Our results do agree with theirs.

\vphantom{}

 Suppose $X$ is a compact metric space. Recall that if $f:X \to X$ is a continuous function, the \textit{inverse limit space} generated by $f$ is 
 

\begin{equation*}
  \underset{\longleftarrow }{\lim }(X,f):=\{(x_{0},x_{1},\ldots ): x_{i} \in X \text{ for each } i \text{ and }
\text{ for each }i\geq 1,x_{i-1}= f(x_{i})\},
\end{equation*}

\noindent which we can abbreviate as $\underset{\longleftarrow }{\lim }f$. The map $f$ on $X$ induces a natural homeomorphism $\sigma$ on $\underset{\longleftarrow }{\lim }f$, which is called the \textit{shift map}, and is defined by 
$$\sigma((x_{0},x_{1}, \ldots))=(x_{1},x_{2},\dots)$$ for $(x_{0},x_{1}, \ldots)$ in  $\underset{\longleftarrow }{\lim }f$.

\vphantom{}

In \cite{Bo2}, R. E. Bowen showed that the topological entropy of the shift map $\sigma$ on $\underset{\longleftarrow }{\lim }f$ is equal to the topological entropy of $f$, where topological entropy of a continuous function on a compact metric space has the original definition due Adler, Konheim, and McAndrew \cite{AKM} and later to Dinaburg \cite{D} and Bowen \cite{Bo}. (While defined differently, these two definitions of topological entropy for a continuous map on a compact metric space coincide.)

\vphantom{}

 Generalized inverse limits, or inverse limits with set-valued functions, are a generalization of (standard) inverse limits. Here, rather than beginning with a continuous function $f$ from a compact metric space $X$ to itself, we begin with an upper semicontinuous function $f$ from $X$ to the closed subsets of $X$. In other words, now our function $f$ is set-valued. The \textit{generalized inverse limit}, or the \textit{inverse
limit with set-valued mappings}, associated with this mapping is the set 
\begin{equation*}
  \underset{\longleftarrow }{\lim }(X,f):=\{(x_{0},x_{1},\ldots ): x_{i} \in X \text{ for each } i, \text{ and }
\text{ for each }i\geq 1,x_{i-1}\in f(x_{i})\},
\end{equation*}

\noindent which is a closed subspace of   $\Pi _{i \ge 0}X$  endowed with the product topology. (As is the case with standard inverse limits, these can be defined in much more general settings, but we do not need those here.) Here again, the shift map $\sigma$ defined above takes $\underset{\longleftarrow }{\lim }(X,f)$ onto itself, but it is no longer a homeomorphism: $\sigma: \underset{\longleftarrow }{\lim }(X,f) \to \underset{\longleftarrow }{\lim }(X,f)$ is a continuous surjection.

\vphantom{}

 The topic of generalized inverse limits  is currently an
intensely studied area of continuum theory, with papers from many authors at this point. See \cite{B}, \cite{BCMM1}, \cite{BCMM2}, \cite{BCMM3}, \cite{BK}, \cite{CR} \cite{GK1}, \cite {GK2}, \cite{Il}, \cite{IM2}, \cite{I1}, \cite{I2}, \cite{I3}, \cite{I4}, \cite{I7},\cite{L}, \cite{M}, \cite{N1},  \cite{N2},  \cite{N3},  \cite{N4}, and \cite{V}, for example. (This list is far from exhaustive, with the number of papers on generalized inverse limits now over 60.)
Tom Ingram and Bill Mahavier included a chapter on these spaces in their
 book \cite{IM1}, and since 
then Tom Ingram has written another book on the topic, 
 \cite{I6}. While most of the research has been on understanding the topology of these spaces, some researchers have recently turned to understanding the dynamical properties, since, for a set-valued map $f:X \to 2^{X}$,  $(\underset{\longleftarrow }{\lim }f, \sigma)$ is a discrete dynamical system. (See \cite{RT}, \cite{KT}, \cite{KN}.)

\vphantom{}

\section{Background and Notation}

Sometimes it is convenient to index our factor spaces, sometimes not. Suppose for each integer $i \ge 0$, $I_{i}=[0,1]=I$. The Hilbert cube is $I^{\infty}=[0,1]^{\infty}=\Pi^{\infty}_{0}I_{i}$. 

\vphantom{}


We often need to talk about various projections from a subset of $I^{\infty}$ into an interval or a product of intervals. Unless it leads to confusion, for a subset $X$ of $I^{\infty}$, and a point $x=(x_{0},x_{1},\ldots)$ in $X$, $\pi_{i}(x)=x_{i}$. (That is, we do not specify the momentary domain of $\pi_{i}$.)  Likewise, if $N$ is a positive integer,  $x=(x_{0},x_{1},..x_{N})$, $x \in X \subset I^{N+1}$, then $\pi_{i}(x)=x_{i}$ for $0 \le i \le N$. Also, we make following definitions.

\begin{itemize}

\item We use both $\mathbb{N}$ and $\mathbb{Z}^{+}$ to denote the positive integers.

\item Let $m \ge 0$ be an integer less than the integer $n$. Then $\langle m,n \rangle = \{m, m+1, \ldots,n\}$, and we call $\langle m,n \rangle$ the \textit{integer interval from} $m$ \textit{to} $n$. Then $\pi_{ \langle m,n \rangle}(x)=(x_{m},x_{m+1},\ldots,x_{n})$. We define $\langle m, \infty \rangle$ to be the set $\{m, m+1, \ldots\}$.

\item Let $ A=\{n_{1},n_{2},\ldots\}$ denote a subset of the nonnegative integers (not necessarily listed in order, and either finite or infinite). Then $\pi_{A}(x)=( x_{n_{1}},x_{n_{2}},\ldots )$.

\item If $A$ is a subset of the space $X$, then $A^{\circ}$ denotes the interior of $A$ in $X$, and $\overline{A}$ denotes the closure of $A$ in $X$.

\item Suppose  $x=(x_{0},x_{1},\ldots,x_{n})$  is a point in $I^{n+1}$ and $y=(y_{0},y_{1},\ldots)$ is a point in $I^{\infty}$. Then we define $x \oplus y$ to be the point $(x_{0},\ldots,x_{n},y_{0},y_{1},\ldots)$ in $I^{\infty}$.  

\item The metric we use on $I^{\infty}$ is  $d(x,y)=\sum^{\infty}_{i=0} \frac{\mid \pi_{i}(x)-\pi_{i}(y) \mid}{2^{i}}$, where $x$ and $y$ are points in $I^{\infty}$.  

\item The \textit{graph} of the set-valued function $f:I \to 2^{I}$ is the set $\Gamma(f)=\{(x,y): y \in f(x) \}$.

\item The set-valued function $f:I \to 2^{I}$ is \textit{upper semicontinuous at the point} $x$ in $I$  if for each open set $V$ in $I$ that contains $f(x)$, there is an open set $U$ in $I$ that contains $x$, and if $z \in U$, then the set $ f(z) \subset V$. The function $f$ is \textit{upper semicontinuous} if it is upper semicontinuous at each point $x$ in $I$. The function $f$ is upper semicontinuous if and only if $\Gamma(f)$ is closed in $I \times I$. (See \cite{A} and \cite{IM1}.)

\item The set-valued function $f:I \to 2^{I}$ is called \textit{surjective} if for each $y \in I$, there is $x \in I$ such that $y \in f(x)$.

\item The \textit{shift map} $\sigma: I^{\infty} \to I^{\infty}$ is defined by $\sigma((x_{0},x_{1},x_{2},\ldots))=(x_{1},x_{2},\ldots)$. The shift map takes $I^{\infty}$ continuously onto itself. Also, if $f:I_{i} \to 2^{I_{i-1}}$ for each $i>0$, and  $M=\underset{\longleftarrow }{\lim }f$, then $\sigma(M) \subset M$. If $f$ is surjective, then $\sigma(M) = M$, and $M$ is invariant under the action of $\sigma$. However, unless $f:I \to I$ is a function, $\sigma$ is not one-to-one.

\item For $A \subset I_{0} \times I_{1}$, define $A^{-1}=\{ (x,y): (y,x) \in A \}$. More generally, if $N$ is a positive integer and $A \subset \Pi_{i=0}^{N} I_{i}$, then $A^{-1}= \{ (x_{N},x_{N-1}, \ldots, x_{1},x_{0}) \in \Pi_{i=0}^{N} I_{i}: (x_{0},x_{1}, \dots, x_{N-1},x_{N}) \in A \}$.

\item Suppose $X,Y$ are topological spaces, and $\mathcal{U}$ is a collection of sets that covers $X$. Then $\mathcal{U} \times Y$ denotes the collection $\{u\times Y: u \in \mathcal{U} \}$, which covers $X \times Y$. 

\item Suppose $\alpha$ is a collection of (open) sets in the space $X$, and $H \subset X$. Then $\alpha \cap H:=\{ A \cap H: A \in \alpha \}$.

\end{itemize}

\subsection*{Topological Entropy Using Open Covers}

For completeness, we review the traditional version of topological entropy (due to Adler, Konheim, and McAndrew \cite{AKM}) and its properties here and follow to a large extent the discussion in Peter Walters' book \cite{Walters book}. We conclude the subsection with theorems on topological entropy due to Bowen \cite{Bo2} that we use. In the next section we recycle and generalize this definition to our new setting.   

\vphantom{}


\textbf{Definitions.}

\begin{itemize}

\item If $\mathcal{U}$ is a finite collection of sets, define $N^{*}(\mathcal{U})$ to be the cardinality of the collection $\mathcal{U}$. If $\mathcal{U}$ is an open cover of the compact topological space $X$, let $N(\mathcal{U})$ denote the number of sets in a finite subcover of $\mathcal{U}$ of smallest cardinality. Define the \textit{entropy} $H(\mathcal{U})$ by $H(\mathcal{U})=\log N(\mathcal{U})$.

\item If $\mathcal{U}$ is a finite collection of open sets that covers the set $G$, then a subcover $\mathcal{U}'$ of $G$ in $\mathcal{U}$ is \textit{minimal} if there does not exist a subcover of $G$ in $\mathcal{U}$ of smaller cardinality.

\item If $\mathcal{U}$ and $\mathcal{V}$ are are open covers of a space $X$, define the \textit{join}  $\mathcal{U} \vee \mathcal{V}$ to be the collection $$\mathcal{U} \vee \mathcal{V}=\{ u \cap v : u \in \mathcal{U}, v \in \mathcal{V} \}$$ of open sets.   The join $\mathcal{U} \vee \mathcal{V}$ is also an open cover of the space $X$. We can likewise define, for a finite collection $\{\mathcal{U}_{i} \}_{i=1}^{n}$ of open covers of $X$, the \textit{join} $\vee_{i=1}^{n} \mathcal{U}_{i}$. 


\item  If $\mathcal{U}$ and $\mathcal{V}$ are are open covers of  the compact topological space $X$, then $\mathcal{U}$ is a \textit{refinement} of $\mathcal{V}$ if each $u \in \mathcal{U}$ is contained in some $v \in \mathcal{V}$. We will say that $\mathcal{V} < \mathcal{U}$ and also that $\mathcal{U} > \mathcal{V}$. Note that if $\mathcal{U}$ is a subcover of $X$ in $\mathcal{V}$, then $\mathcal{U}$ is both a subcollection of $\mathcal{V}$ and a  refinement of $\mathcal{V}$, and $\mathcal{V} < \mathcal{U}$.

\item If $X$ is a compact topological space, $\alpha$ is an open cover of $X$, and $T:X \to X$ is continuous, then $T^{-1}(\alpha)$ is the open cover consisting of all sets $T^{-1}(A)$ where $A \in \alpha$.  Also, $$T^{-1}(\alpha \vee \beta)=T^{-1}(\alpha) \vee T^{-1}(\beta),$$ and 
$$\alpha < \beta \text{ implies } T^{-1}(\alpha) < T^{-1}(\beta).$$

We denote $\alpha \vee T^{-1}(\alpha) \vee \cdots \vee T^{-n}(\alpha)$ by $\vee_{i=0}^{n}T^{-i}(\alpha)$.

\end{itemize}

\vphantom{}

\noindent \textbf{Remarks.} Suppose $\alpha$ and $\beta$ are open covers of the compact topological space $X$. Then 

\begin{enumerate}

\item $H(\alpha) \ge 0$.

\item $H(\alpha)=0$ if and only if $N(\alpha)=1$ if and only if $X \in \alpha$.

\item If $\alpha < \beta$, then $H(\alpha) \le H(\beta)$.

\item $H(\alpha \vee \beta) \le H(\alpha) + H(\beta)$.

\item If $T:X \to X$ is a continuous map, then $H(T^{-1}(\alpha)) \le H(\alpha)$. If $T$ is also surjective, then $H(T^{-1}(\alpha)) = H(\alpha)$.

\end{enumerate}

(See \cite{Walters book} for proofs of (3), (4), and (5) above.)

\vphantom{}

We will need the following lemma, which is used in the proof of Theorem \ref{ent lim}, and in our results.

\begin{lemma} \label{Walters} \cite {Walters book} If $\{ a_{n} \}_{n \ge 1}$ is a sequence of nonnegative real numbers such that $a_{n+p} \le a_{n} +a_{p}$ for each $n,p \in \mathbb{N}$, then $\lim\limits_{n \to \infty} \frac{a_{n}}{n}$ exists and equals $ \underset{n}{\inf}\frac {a_{n}}{n}$.

\end{lemma}

\begin{theorem} \label{ent lim} (See \cite{Walters book}.) If $\alpha$ is an open cover of $X$ and $T:X \to X$ is continuous, then $\lim_{n \to \infty} \frac{H(\vee_{i=0}^{n}T^{-i}(\alpha))}{n}$ exists.

\end{theorem}


\vphantom{}

\textbf{Definition.} If $\alpha$ is an open cover of the compact topological space $X$, and $T:X \to X$ is continuous, then the \textit{entropy of} $T$ \textit{relative to} $\alpha$ is $h(T,\alpha)$ given by
$$h(T,\alpha)=\lim_{n \to \infty} \frac{H(\vee_{i=0}^{n}T^{-i}(\alpha))}{n}.$$

\textbf{Remarks.} 
\begin{enumerate}

\item $h(T,\alpha) \ge 0$.
\item If $\alpha < \beta$, then $h(T,\alpha)\le h(T,\beta)$.
\item $h(T,\alpha) \le H(\alpha)$.

\end{enumerate}

(See \cite{Walters book} for proofs of (1), (2), and (3) above.)


\vphantom{}

\textbf{Definition.} If $T:X \to X$ is continuous, the \textit{topological entropy} $h(T)$ of $T$ is given by 
$$ h(T)=\sup_{\alpha} h(T,\alpha)$$
where $\alpha$ ranges over all open covers of $X$. 

\vphantom{}

\vphantom{}

\textbf{Remarks.}

\begin{enumerate}

\item $\infty \ge h(T) \ge 0$. 
\item In the definition of $h(T)$ one can take the supremum over finite open covers of $X$. This follows from the fact that if $\alpha < \beta$, then $h(T,\alpha)\le h(T,\beta)$.

\item If $id_{X}$ denotes the identity map from $X$ to $X$, then $h(id_{X})=0$.

\item If $Y$ is a closed subset of $X$ and $T(Y)=Y$, then $h(T|Y) \le h(T)$.

\end{enumerate}

\vphantom{}

\begin{theorem} (See \cite{Walters book}.) If $X_{1},X_{2}$ are compact spaces and $T_{i}:X_{i} \to X_{i}$ are continuous for $i=1,2$, and if $\phi:X_{1} \to X_{2}$ is a continuous map with $\phi(X_{1})=X_{2}$ and $\phi \circ T_{1}=T_{2} \circ \phi$, then $h(T_{1}) \ge h(T_{2})$. If $\phi$ is a homeomorphism, then $h(T_{1}) = h(T_{2})$.

\end{theorem}

\begin{theorem}  (See \cite{Walters book}.) If $T:X \to X$ is a homeomorphism of a compact space $X$, then $h(T)=h(T^{-1})$.

\end{theorem}

\begin{theorem} (See \cite{Walters book}, Thm. 7.10.) If $T:X \to X$ is a continuous map of a compact metric space $X$, then $h(T^{n})=n \thinspace h(T)$.

\end{theorem}

\begin{theorem} (\cite{Bo2}, Proposition 5.2) Suppose $f:X\to X$ is a continuous surjective map on a compact Hausdorff space. If $\sigma$ denotes the induced shift homeomorphism on the inverse limit space $\underset{\longleftarrow }{\lim }(X,f)$, then $h(\sigma)=h(f)$.

\end{theorem}

 Suppose $X$ is a compact metric space and $f:X \to X$ is continuous. A point $x$ in $X$ is called a \textit{wandering point} if there is an open set $U$ containing $x$ such that $U \cap (\cup_{m \ne 0, m \in \mathbb{Z}} f^m(U))=\emptyset$. (This is the definition Bowen gives in   \cite{Bo2}.)

\vphantom{}
Today most authors use this definition: (\cite{Walters book}) A point $x$ is called \textit{wandering} for $f$ if there is an open set $U$ containing $x$ such that the sets $f^{-n}(U)$, $n \ge 0$, are mutually disjoint. The proposition below shows that these two definitions are equivalent. (Surely this is known, but do not know where this is shown, so we give a proof.)

\begin{prop}
	Bowen's and Walters' definition of a wandering point are equivalent. 
\end{prop}
\begin{proof}
	Suppose $X$ is a compact metric space and $f:X \to X$ is continuous.
	
	Let us prove that $\Omega_f^B=\Omega_f^W.$
	
	Suppose $x \in \Omega_f^W.$ Then for every neighbourhood $U$ of $x$ exists $n \geq 1$ such that $f^{-n}(U) \cap U \neq \emptyset.$ But then $ \cup_{m \notin \{0, n\},  m \in \mathbb{Z}} (f^{-m}(U) \cap U) \cup (f^{-n}(U) \cap U) = U \cap (\cup_{m \ne 0, m \in \mathbb{Z}} f^m(U)) \neq \emptyset.$ Hence, $x \in \Omega_f^B$ and so $\Omega_f^W \subseteq \Omega_f^B.$
	
	Let us prove that $\Omega_f^B \subseteq \Omega_f^W.$ Suppose $x \in \Omega_f^B$ and $U$ arbitrary open set containing $x$. Then $U \cap (\cup_{m \ne 0, m \in \mathbb{Z}} f^m(U))\ne\emptyset,$ hence $\cup_{m \ne 0, m \in \mathbb{Z}} (U \cap f^m(U)) \ne\emptyset.$ Therefore, exists $m \in \mathbb Z$ such that $U \cap f^m(U) \ne\emptyset.$ If $m<0$ statement is true, so let us assume that $m>0.$ Then there is some $y \in U \cap f^m(U).$ Since $y \in f^m(U), $ there is $x' \in U$ such that $f^m(x')=y.$ But then, since $f^m(x')=y \in U,$ we have $x' \in {(f^m)}^{-1}(U)=f^{-m}(U)$ so $x' \in U\cap f^{-m}(U) \neq \emptyset $ and therefore $x \in \Omega_f^W.$
	
	Hence, $\Omega_f^B=\Omega_f^W.$
	
\end{proof}

We will call a nonempty open set $U$ with the property that $U \cap (\cup_{m \ne 0, m \in \mathbb{Z}} f^m(U))=\emptyset$ a \textit{simple wandering set}.

\vphantom{}

If a point in $X$ is not a wandering point, we call it a \textit{nonwandering point}. The set of nonwandering points of $X$ under $f$ is denoted $\Omega_f$ or just $\Omega$ if no ambiguity results. Thus, 

$$ \Omega_f= \{ x \in X: \text{ for each open set } U \text{ containing } x,  U \cap (\cup_{m \ne 0, m \in \mathbb{Z}} f^m(U))\ne\emptyset \}.$$

\vphantom{}

The nonwandering set $\Omega_f$ is closed and invariant under $f$, in the sense that $f(\Omega_f) \subset \Omega_f$ (\cite{Walters book}, Theorem 5.6, p 124). Hence, the set of wandering points $X \setminus \Omega_f$ is open in $X$. 

\vphantom{}

\begin{theorem} \label{Bowen1} (See \cite{Bo2})
	Let $f:M\to M$ be a continuous map on a compact metric space. If $\Omega$ is the wandering set of $f,$ then $h(f)=h(f)|\Omega.$
\end{theorem}

\subsection*{Mahavier Products}

The Mahavier product is a useful tool for studying subsets of a generalized inverse limit. (See \cite{GK2} and \cite{BK}.) They have nice topological and algebraic properties. A closed subset of a standard inverse limit (where the bonding maps are single valued functions - so actually a function from the space to itself) is a subinverse limit. \textit{ Unfortunately, a closed subset of a generalized inverse limit need not be a ``subgeneralized inverse limit''}. The Mahavier product allows one to consider closed subsets of the generalized inverse limit space, whether or not they are subgeneralized inverse limits. It also makes it easy to consider ``finite'' generalized inverse limits and their subsets. Suppose $n \ge 2$. Then $\{(x_{0},x_{1},\ldots,x_{n}): x_{i-1} \in f(x_{i}) \text{ for } n \ge i >0 \}$ is what we mean by a \textit{finite generalized inverse limit}. (For standard inverse limits over intervals, these sets are always arcs topologically, and therefore of limited interest. This is not the case for generalized inverse limits.)


\vphantom{}

Define the \textit{Mahavier product} \cite{GK2} as follows:

Suppose $X$, $Y$, and $Z$ are sets. Suppose $A \subset X \times Y$ and $B \subset Y \times Z$. Define 
\begin{equation*}
   A \star B =\{(x,y,z ): (x,y) \in X \times Y, \text{ and } (y,z) \in Y \times Z \}
\end{equation*}

\noindent to be the \textit{Mahavier product} of $A$ and $B$. Thus, $A \star B \subset X \times Y \times Z$. If $A=\{(a,b) \}$ and $B=\{(b,c) \}$, then we write $(a,b) \star (b,c)$ to mean $A \star B=\{(a,b) \} \star \{(b,c) \} $. Note that $A \star B=(A \times Z) \cap (X \times B)$. In the definition, the factor space $Y$ is acting as the ``link up'' factor; that is, if $a \in A$, then there is a point $c \in A \star B$ such that the projection of $c$ to $X \times Y$  is $a$ only if some point $b \in B$ has its projection to $Y$ the same as the projection of $a$ to $Y$. In particular, if the projection of $A$ to $Y$ does not intersect the projection of $B$ to $Y$, then $A \star B = \emptyset$.

\vphantom{}

If, in addition, $W$ is a set and $C \subset Z \times W$, we can form the Mahavier products $(A \star B) \star C$ and $A \star (B \star C)$. One of the nice algebraic properties of Mahavier products is that they are associative. Hence, $(A \star B) \star C=A \star (B \star C)$, and we can just write $A \star B \star C$ without ambiguity. Also, if $A,B$ are open (closed), then $A\star B$ is open (closed).

\vphantom{}

Now suppose that for each $i \ge 0$, $X_{i}$ is a set, and $A_{i} \subset  X_{i-1} \times  X_{i}$. Then, using induction, we can form the Mahavier products 
$$A_{1} \star A_{2} \star \cdots \star A_{n} \subset X_{0} \times \cdots X_{n} $$ 
for each positive integer $n$. We often write $\star_{i=1}^{n}A_{i}$ for $A_{1} \star A_{2} \star \cdots \star A_{n}$. If $A_i=A,  1 \leq i \leq n $ then $$\overbrace{ A \star A \star A \star \cdot \cdot \cdot \star A}^{n}= \star_{i=1}^{n}A.$$ We also define $A_{1} \star A_{2} \star \cdots$ to be the set of points $x=(x_{0},x_{1}, \ldots) \in \Pi_{i=0}^{\infty}X_{i}$ such that $(x_{i-2},x_{i-1},x_{i}) \in A_{i-1} \star A_{i}$ for each $i \ge 2$,  and we write $\star_{i=1}^{\infty}A_{i}$ for  $A_{1} \star A_{2} \star \cdots$. Note that if $A_i$ is closed in $X_{i-1} \times X_{i}$, then $\star_{i=1}^{\infty}A_{i}$ is closed in $\Pi_{i=0}^{\infty}X_i$.

\vphantom{}

Note that if $f:I \to 2^{I}$ is a bonding function with graph $\Gamma(f)$, and $G:=(\Gamma(f))^{-1}= \{(y,x): (x,y) \in \Gamma(f) \}$, then 

\begin{equation*}
  M=\underset{\longleftarrow }{\lim } \text{ } f= G \star G \star G \star \cdots = \star^{\infty}_{i=1}G.
\end{equation*}

\vphantom{}

In the sequel, we consider Mahavier products of the form $G\star H$ when $G\subset I^N$ and $H\subset I^N.$ There is some ambiguity about what the set $G\star H$ is in this situation. In this paper we always mean the following: Write $I^N$ as $I^{N-1}\times I$ for $G$ and $I^N$ as $I\times I^{N-1}$ for $H.$ Then $G\subset I^{N-1}\times I$ and $H\subset I\times I^{N-1}$ and
$$
G\star H = \left\{  (x,y,z):  (x,y)\in G \text{ and } (y,z)\in H, \text{ where } x\in I^{N-1}, y\in I, z \in  I^{N-1}  \right\}.
$$

\vphantom{}



\section{Preliminary Results}

While we are mostly interested in closed sets that are the graphs of upper semicontinuous functions from $I$ to $2^I$, we do not need for our closed set to be such a graph in order to define its topological entropy. Also, often, the entropy of such a graph is determined by a closed subset of the graph - sometimes a finite subset of the graph. Moreover, in order to discuss the entropy properties of closed subsets of $I^2$, we need to define the topological entropy of closed subsets of $I^n$ for $n$ a positive integer greater than $1$.

\vphantom{}

Before we define topological entropy for closed subsets of $[0,1]^n$, we need some background information on the closed sets and open covers we are using.

\vphantom{}

The following examples demonstrate that if $G$ is a closed subset of $[0,1]^{n}$, $n>1$ , then (1) it may be the case that $\mathbf{G}:= \star_{i=1}^{\infty}G= \emptyset$ (and that $G$ is of limited interest), and (2) even if $\mathbf{G}:= \star_{i=1}^{\infty}G \ne \emptyset$, it may be the case that $\sigma(\mathbf{G}) \ne \mathbf{G}$.

\vphantom{}

\begin{exam}
Suppose $G$ is the closed subset $[\frac{2}{3},1] \times [0,\frac{1}{3}]$ of $I_{0} \times I_{1}$. Then $\mathbf{G}:= \star_{i=1}^{\infty}G= \emptyset$. In fact, $G \star G= \emptyset$. (It is equally easy to construct empty examples in higher dimensions.)

\end{exam}

\begin{exam}
Let $L_{0}=I_{0} \times \{p\}$ and $L_{1}=I_{0} \times \{q\}$, where $0 \le p < q \le 1$.Suppose $G$ is the closed subset $L_{0} \cup L_{1}$ of $I_{0} \times I_{1}$. Then $\mathbf{G}:= \star_{i=1}^{\infty}G$ is a Cantor set of arcs, and $\sigma(\mathbf{G})$ is a Cantor set and is a proper subset of $\mathbf{G}$. 

\vphantom{}

\begin{proof}
Let $C=\{s=(s_{1},s_{2}, \ldots):s_{i} \in \{p,q\} \text{ for each } i>0 \}$. Then $C$ is a Cantor set contained in $\mathbf{G}$. Moreover, for each $s \in C$, $I_{0} \times \{s\}$ is an arc contained in $\mathbf{G}$, and $\mathbf{G}=\cup \{I_{0} \times \{s\}: s \in C \}$. Hence, $\mathbf{G}$ is a Cantor set of arcs. Since $\sigma(\mathbf{G})=C$, $C$ is a proper subset of $\mathbf{G}$.

\end{proof}

\end{exam}

\begin{exam} Suppose $0 \le p <q \le 1$, $L_p=I_0 \times I_1 \times \{p\}$, and $L_q=I_0 \times I_1 \times \{q\}$. Then if $s,t \in \{p,q\}$, $L_s \star L_t=I^2 \times \{s\} \times I \times \{t\}$. If $s=s_1,s_2, \ldots$ is a sequence each member of which is either $p$ or $q$, let $M_s=L_{s_1} \star L_{s_2} \star \cdots=I^2 \times \{s_1\} \times I \times \{s_2\} \times I \times \{s_3\} \cdots$. If $G = L_p \cup L_q$, then $G$ is a closed subset of $I^3$, and $\star_{i=1}^{\infty}G=\cup \mathcal{M}$, where $\mathcal{M} = \{M_s: s=s_1,s_2, \ldots \text{ is a sequence each member of which is either } p \text{ or } q \}$. Each $M_s$ is homeomorphic to the Hilbert cube, so if $M = \star_{i=1}^{\infty}G$, then $M$ is topologically a Cantor set of Hilbert cubes (with $M_s \cap M_t = \emptyset$ when $s,t$ are different sequences of $p$'s and $q$'s). Note that because $G \subset I_0 \times I_1 \times I_2$, $\sigma(M)$ is not a subset of $M$, but $\sigma^2(M) \subset M$. Let $M'=\cup\{s_1 \times I \times s_2 \times I \times s_3 \times \cdots : s_1,s_2,\ldots \text{ is a sequence each member of which is either  } p \text{ or } q\}$. Then $\sigma^2(M) =M' \subset M$, but $M' \ne M$. However, $\sigma^2(M')=M'$.

\end{exam}

\vphantom{}

The following propositions give some natural conditions under which $\star_{i=0}^{\infty} G \ne \emptyset$ for $G$ a closed subset of $[0,1]^n$ ($n>1$).

\begin{prop} \label{prop5.31} If $G$ is a nonempty closed subset of $I_{0} \times I_{1}$, then $\mathbf{G}=\star_{i=1}^{\infty}G \ne \emptyset$ if and only if for every integer $m \geq 2,$ $\star_{i=1}^{m}G \neq \emptyset$. If $G$ is a nonempty closed subset of $I_0\times I_1 \times \cdots \times I_n,$ then $\mathbf{G}=\star_{i=1}^{\infty}G \ne \emptyset$ if and only if for every integer $m \geq 2,$ $\star_{i=1}^{mn}G \neq \emptyset.$
\end{prop}


\begin{proof}
	If $\mathbf{G}=\star_{i=1}^{\infty}G \ne \emptyset$ it follows from the definition that $\star_{i=1}^{m}G \neq \emptyset, \forall m\in \mathbb N.$
	Now, suppose $\star_{i=1}^{m}G \neq \emptyset$ for every integer $m \geq 2.$ We will inductively define a point in $\mathbf{G}=\star_{i=1}^{\infty}G.$
	First observe the following: 
	If $(x_0,\ldots, x_{m-1},x_m) \in \star_{i=1}^{m}G $ for some integer $m\geq 2$ then $(x_0,\ldots, x_{m-1}) \in \star_{i=1}^{m-1}G. \quad (*)$

	For $m=2$ we have $G \star G \neq \emptyset$  so from the above it follows that there is a point $(x,y) \in G$ and $z \in [0,1]$ such that $(x,y,z) \in G \star G.$
	
	Now, for given $m=k$ we have that $ \star_{i=1}^k G \neq \emptyset.$ So, there are points $(x_0,x_1,\ldots,x_k)$ $ \in \star_{i=1}^k G$ and $x_{k+1} \in [0,1]$ such that $(x_0,x_1,\ldots,x_k,x_{k+1}) \in \star_{i=1}^{k+1} G.$ This follows from $(*)$ and assumption $ \star_{i=1}^{k+1} G \neq \emptyset.$
	
	Therefore we have constructed a sequence $x_0,x_1,\ldots, x_k,\ldots,$ such that for each positive integer $m\geq 2, (x_0,x_1,\ldots,x_m) \in  \star_{i=1}^m G $ i.e. $(x_0,x_1,\ldots,x_k,\ldots) \in \star_{i=1}^{\infty}G.$
	Therefore we have $\mathbf{G}=\star_{i=1}^{\infty}G \ne \emptyset.$
	The proof of the second statement is similar so we omit it.
\end{proof}

 \begin{prop}
 	Let $G$ be a nonempty closed subset of $I_{0} \times I_{1}.$ If there is some point $(x,y)\in G$ such that $(y,x)$ is also in $G,$ then $\mathbf{G} =\star_{i=1}^{\infty}G \ne \emptyset.$ If $G$ is a nonempty closed subset of $\Pi_{i=0}^n I_i$ and if there is some point $(x_0,\ldots,x_n)\in G$ such that $(x_n,\ldots,x_0)$ is also in $G,$ then  $\mathbf{G} =\star_{i=1}^{\infty}G \ne \emptyset.$
 \end{prop}
 \begin{proof}
 	Let $(x,y)\in G$ such that $(y,x)\in G.$ Then, for each integer $m\geq 2$ we have that $(x,y,x,y,\ldots,x) \in \star_{i=1}^{m}G$ if $m$ is even or $(x,y,x,\ldots,y) \in \star_{i=1}^{m}G$ if $m$ is odd. In both cases $\star_{i=1}^{m}G \ne \emptyset$ and therefore, from the previous proposition it follows that $\mathbf{G}=\star_{i=1}^{\infty}G \ne \emptyset.$ 
 	
 	The proof of the second statement is similar and we omit it.
 \end{proof}

 	\begin{corollary} If $G=G^{-1}\neq \emptyset,$ where $G^{-1}=\left \{ (y,x): (x,y) \in G \right  \}$ is a nonempty closed subset of $I_{0} \times I_{1}$, then $\mathbf{G}=\star_{i=1}^{\infty}G \ne \emptyset$. 
 		
 	\end{corollary}
 	
 	\begin{proof} Since $G \ne \emptyset$, there is some point $(x,y) \in G$. Since $G=G^{-1}$, the point $(y,x) \in G$. From previous proposition it follows that $\mathbf{G} \ne \emptyset$.
 	\end{proof}
 	
 	\begin{prop}
 	If $G$ is a nonempty closed subset of $\Pi_{i=0}^n I_i$ and $\pi_n(G) \subset \pi_0(G)$, then $\mathbf{G}=\star_{i=1}^{\infty}G \ne \emptyset$.
 	\end{prop}
 	
 \begin{proof}
 	The set $G \ne \emptyset$, so there is some $(x_0,\ldots,x_n) \in G$. Since $\pi_n(G) \subset \pi_0(G)$, $x_n \in \pi_0(G)$, there are points $x_{n+1}, x_{n+2},\ldots,x_{2n}$ in $I$ such that $(x_n,x_{n+1},\ldots,x_{2n}) \in G$. Thus $(x_0,\ldots,x_n,\ldots,x_{2n})$ $\in G \star G$. Then since $x_{2n} \in \pi_0(G)$, there is $(x_{2n},\ldots,x_{3n}) \in G$ such that $(x_0,\ldots,x_n,\ldots,$ $x_{2n},\ldots,x_{3n}) \in G \star G \star G$, and we can continue this process indefinitely, obtaining a point in $\mathbf{G}$.
 \end{proof}

\begin{prop} \label{Gnonempty} Suppose $n$ is a positive integer. If $G$ is a closed subset of $I_{0} \times I_{1}$ that contains a finite set of  points $\{(x_{0},x_{1}),(x_{1},x_{2}), \ldots , (x_{n-1},x_{n}),(x_{n},x_{0} )\}$, then $\mathbf{G}=\star_{i=1}^{\infty}G \ne \emptyset$. Furthermore, $\mathbf{G}$ contains a point of period $n$ under the action of $\sigma$.

\end{prop}

\begin{proof} The point $(x_{0},x_{1}, \ldots, x_{n},x_{0}, \ldots, x_{n}, \ldots) \in \mathbf{G}$, so $\mathbf{G} \ne \emptyset$. Let $y_{0}=(x_{0}, \dots, x_{n})$ and $y_{n}=(x_{n},x_{0}, \ldots x_{n-1})$. For each $0 < i <n$, let $y_{i}=(x_{i}, \ldots, x_{n},x_{0}, \ldots x_{i-1})$. For $0 \le i \le n$, let $z_{i}=y_{i} \oplus y_{i} \oplus y_{i} \oplus \ldots $ . Then each $z_{i} \in \mathbf{G}$, and $\sigma(z_{i})=z_{i+1}$ for $0 \le i <n$, and $\sigma(z_{n})=z_{0}$. Hence, $\sigma^{n}(z_{0})=z_{0}$.

\end{proof}

\begin{prop} 
 If $G$ is a nonempty closed subset of $I_{0} \times I_{1}$ and $\mathbf{G}=\star_{i=1}^{\infty}G \ne \emptyset$, then $\sigma(\mathbf{G}) \subset \mathbf{G}$.  If $G$ is a nonempty closed subset of $I_{0} \times I_{1} \times \cdots \times I_n$ and $\mathbf{G}=\star_{i=1}^{\infty}G \ne \emptyset$, then $\sigma^n(\mathbf{G}) \subset \mathbf{G}$. 

\end{prop}


\begin{proof}
We prove the first statement. The proof for the second statement is similar. Suppose $x=(x_{0},x_{1}, \dots) \in \sigma(\mathbf{G})$. Then there is $y=(y_{0},y_{1}, \ldots) \in \mathbf{G}$ such that $\sigma(y)=x$. Now $\sigma(y) = (y_{1},y_{2}, \ldots)=x$, so $x_{i-1}=y_{i}$ for each $i>0$. Since $y \in \mathbf{G}$, for each $i>0$, $(y_{i-1},y_{i}) \in G$. Then for each $i>1$, $(y_{i-1},y_{i})=(x_{i-2},x_{i-1}) \in G$. Then $x \in \mathbf{G}$.
\end{proof}

Proposition \ref{Gnonempty} can be generalized to $G\subset I^{n+1}, n>1$, too. First we give an example, then state the more general proposition.

\begin{exam}
	Suppose $G$ is a nonempty and closed subset of $ I^{2+1}$ such that \\
	$\left\{ (x_0,x_1,x_2), (x_2,x_3,x_4), (x_4,x_5,x_6) \right\} \subset G$ and $x_6=x_0.$ Then $(x_0,x_1,x_2,x_3,x_4,x_5,x_6)=(x_0,x_1,x_2,x_3,x_4,x_5,x_0) \in G\star G\star G$ and the point $y_0=(x_0,x_1,x_2,x_3,x_4,x_5) \oplus (x_0,x_1,x_2,x_3,x_4,x_5) \oplus \ldots \in \star_{i=1}^{\infty}G =\mathbf{G} $ and $\mathbf{G} \ne \emptyset.$
	Now $\sigma^2 (\mathbf{G}) \subset \mathbf{G}$ and 
	\begin{align*}
	&\sigma^2(y_0)=(x_2,x_3,x_4,x_5,x_0,x_1,\ldots,x_5,x_0,\ldots):=y_1\\
	&\sigma^2(y_1)=(x_4,x_5,x_0,x_1,\ldots,x_5,x_0,\ldots):=y_2\\
	&\sigma^2(y_2)=(x_0,x_1,\ldots,x_5,x_0,\ldots)=y_0
	\end{align*}
Therefore, $y_0$ is a point in $\mathbf{G}$ and $\left( \sigma^2 \right)^3(y_0)=y_0.$ So $y_0$ is a point in $\mathbf{G}$ of period $3$ under the action of $\sigma^2.$
\end{exam}
\begin{prop}
	Suppose $G$ is a closed subset of $I^{n+1}, n\in \mathbb N,$ and $p\in \mathbb N,p>1.$ If $y_i=(x_{in},\ldots,x_{(i+1)n})\in G$ for $0\leq i\leq p-1,$ and $x_{pn}=x_0,$ then $ z_0= y_0 \star y_1 \star \ldots \star y_{p-1} \star y_0 \star y_1 \star \ldots \star y_{p-1} \star \ldots \in \mathbf{G}=\star_{i=1}^{\infty}G \ne \emptyset$, and $z_0$ is a period $p$ point under the action of $\sigma^n$ in $\mathbf{G}.$
\end{prop}
\begin{proof}
	The proof is straightforward and we omit it.
\end{proof}


\begin{prop} \mbox{}

\begin{enumerate}

\item  Let $G$ be a nonempty closed subset of $I_0 \times I_1$ such that $\mathbf{G} \ne \emptyset$. Let $\cap_{i=0}^{\infty} \sigma^{i}(\mathbf{G})=\mathbf{G}^{*}$. Then $\mathbf{G}^{*} \ne \emptyset$ and $\mathbf{G}^{*} \subset \mathbf{G}$. Furthermore, $\sigma(\mathbf{G}^{*}) = \mathbf{G}^{*}$.

\item Let $G$ be a nonempty closed subset of $I_0 \times I_1 \times \cdots \times I_n$ such that $\mathbf{G} \ne \emptyset$. Let $\cap_{i=0}^{\infty} \sigma^{in}(\mathbf{G})=\mathbf{G}^{*}$. Then $\mathbf{G}^{*} \ne \emptyset$ and $\mathbf{G}^{*} \subset \mathbf{G}$. Furthermore, $\sigma^n(\mathbf{G}^{*}) = \mathbf{G}^{*}$.

\end{enumerate}

\end{prop}  

\begin{proof} We give a proof of (1). The proof of (2) is similar.
Since $\sigma^{n}(\mathbf{G}) \subset \sigma^{n-1}(\mathbf{G})$ for $n>0$, $\mathbf{G}^{*} \ne \emptyset$ and $\mathbf{G}^{*} \subset \mathbf{G}$. It is also easy to see that  $\sigma(\mathbf{G}^{*}) = \mathbf{G}^{*}$. 
 
\end{proof}

For  $G$ is a nonempty closed subset of $I_{0} \times I_{1} \times \cdots \times I_n$ and $\mathbf{G}=\star_{i=1}^{\infty}G \ne \emptyset$, we will call the set  $\mathbf{G}^{*} = \cap_{i=0}^{\infty} \sigma^{in}(\mathbf{G})$ the \textit{kernel} of $\mathbf{G}$.

\subsection*{Grid covers} 
Suppose $K$ is a closed subset of $I^{\infty}$. Let $\tau=\{\tau_{1}, \dots, \tau_{n}\}$ be a minimal open cover of $[0,1]$ by open intervals. Let $N$ be a positive integer. The \textit{grid generated by} $\tau$ \textit{for} $N$ is the collection $T$ of basic open sets in $I^{\infty}$ $$T=\{ \tau_{i_{0}} \times \tau_{i_{1}} \times \ldots \times \tau_{i_{N}} \times I^{\infty}: i_{j} \in \langle 1,n \rangle \}.$$ Since $T$ is an open cover of $I^{\infty}$ by basic open sets, it is therefore also a cover of $K$ by basic open sets. We will say that $T$ is a \textit{grid cover} of $K$. Likewise,  $$S=\{ \tau_{i_{0}} \times \tau_{i_{1}} \times \ldots \times \tau_{i_{N}}: i_{j} \in \langle 1,n \rangle \}.$$ is a \textit{grid cover} of $I^{N+1}$ by basic open sets, and is also therefore a cover of any closed subset $L$ of $I^{N+1}$.


Surely the following propositions are known, but we include them just to make sure our grid covers ``do the job'' that \textit{any} open cover of a compact subset of $I^n$ or $I^{\infty}$ would do.

\begin{prop} \label{grid cover}
Suppose $K$ is a closed subset of $I^{\infty}$. If $\mathcal{U}$ is an open cover of $K$ by open sets in $I^{\infty}$, then there is a grid cover $T$ of $I^{\infty}$ such that $T'=\{o \in T: o \cap K \ne \emptyset \}$ refines $\mathcal{U}$ and covers $K$. If $\mathcal{V}$ is an open cover of $K$ by open sets in the subspace $K$, then there is a grid cover $T$ of $I^{\infty}$ such that $T^{*}=\{o \cap K : o \in T \}$ refines $\mathcal{V}$ and covers $K$.

\end{prop} 

\begin{proof} Suppose $\mathcal{U}$ is an open cover of $K$ by open sets in $I^{\infty}$. Then there is a collection $\mathcal{V}$ of basic open sets in $I^{\infty}$ that refines $\mathcal{U}$ and covers $K$. We also assume that for each $v \in \mathcal{V}$, each projection $\pi_{k}(v)$ is an open interval (relative to $[0,1]$). Since $K$ is compact, there is a finite subcover $\mathcal{V}'$ of $\mathcal{V}$. Let $\mathcal{V}'=\{v_{1}, \ldots, v_{m} \}$.  There is a collection $\mathcal{W}$ of basic open sets such that $\mathcal{W}=\{ w_{1}, \ldots, w_{m} \}$, $\overline{w_{i}} \subset v_{i}$ for $1 \le i \le m$, and $\mathcal{W}$ covers $K$. Again, we can choose the collection $\mathcal{W}$ so that for each $w \in \mathcal{W}$, each projection $\pi_{k}(w)$ is an open interval relative to $[0,1]$. Since each $w_{i}$ is a basic open set, there is some positive integer $N$ such that $\pi_{j}(w_{i})=I$ for each $j > N$.

\vphantom{}

Let $\pi_{k}(\overline{w_{i}})=[a_{i,k},b_{i,k}]$ for each $0 \le k \le N$, $1\le i \le m$, and let 

$$ \mathcal{E}=\{x: x \in \{a_{i,k},b_{i,k}\}, 0 \le k \le N, 1 \le i \le n \}\cup \{0,1\}.$$ Since $\mathcal{E}$ is a finite subset of $I$, we can list the members of $\mathcal{E}$ in increasing order as $\mathcal{E}=\{0=t_{0},t_{1}, \dots, t_{\gamma}=1 \}$. Then each $\pi_{k}(\overline{w_{j}})$ is a unique union of consecutive intervals of the form $[t_{i-1},t_{i}]$. If $x=(x_{0},x_{1}, \ldots ) \in K$, there is some $w_{i}$ such that $x \in w_{i}$, which implies that for each $k \le N$, $x_{k} \in \pi_{k}(\overline{w_{i}})$, and there is some $t_{j_{k}}$ such that $x_{k} \in [t_{j_{k}},t_{j_{k}+1}]$. Thus, $x \in \Pi_{k=0}^{N}[t_{j_{k}},t_{j_{k}+1}] \times I^{\infty}$.

\vphantom{}

Suppose $\epsilon>0$. Let $\pi_{k}(\overline{w_{i}})^{+}=(a_{i,k}-\epsilon,b_{i,k}+\epsilon) \cap [0,1]$ for each $0 \le k \le N$, $1\le i \le m$.  Let $w_{i}^{+}=\Pi_{k=0}^{N}\pi_{k}(\overline{w_{i}})^{+} \times I^{\infty}$. We can choose $\epsilon >0$ so small that (1) $\epsilon < \min_{i=0}^{\gamma -1} \frac{ \{|t_{i+1}-t_{i}|\}}{16}$, and (2) $\overline{w_{i}} \subset w_{i}^{+} \subset v_{i}$. Then each $w_{i}^{+}$ is a union of members of $\Gamma = \{ \Pi_{i=0}^{N}((t_{j_{i}}-\epsilon,t_{j_{i}+1}+ \epsilon)\cap [0,1]) \times I^{\infty}: j_{i} \in \langle 0, \gamma \rangle \}$. Hence, if $\Gamma^{*}= \{ g \in \Gamma: g \subset \overline{w_i}^+ \text{ for some } i \text{ and } g \cap K \ne \emptyset \}$, then $\Gamma^{*} > \mathcal{W}> \mathcal{V}> \mathcal{U}$ and $\Gamma^{*}$ covers $K$.

\vphantom{}

The proof of the last statement now follows easily, so we omit it.

\end{proof}


\begin{prop} 
Suppose $M$ is a positive integer and $K$ is a closed subset of $I^{M+1}$. If $\mathcal{U}$ is an open cover of $K$ by open sets in $I^{M+1}$, then there is a grid cover $T$ of $I^{M+1}$ such that $T'=\{o \in T: o \cap K \ne \emptyset \}$ refines $\mathcal{U}$ and covers $K$. If $\mathcal{V}$ is an open cover of $K$ by open sets in the subspace $K$, then there is a grid cover $T$ of $I^{M+1}$ such that $T^{*}=\{o \cap K : o \in T \}$ refines $\mathcal{V}$ and covers $K$.

\end{prop} 
 
\begin{proof} The proof is similar to the proof of Proposition \ref{grid cover} and we omit it.

\end{proof}

\vphantom{}

For a grid cover $T$ of $I^{M+1}$ or $I^{\infty}$, we refer to the members of $T$ as \textit{boxes}. Setting up the machinery for a definition of topological entropy of $G$, a closed subset of $[0,1]^{2}$ (and later for $G$ a closed subset of $I^{M+1}$), takes some doing, but once in place, we will be able to compute topological entropy by ``counting'' the boxes our relevant sets intersect.


\section{Topological entropy of closed subsets of $[0,1]^{2}$}

We index our intervals for bookkeeping purposes. For convenience, we also write $I^{\infty}$ for $\Pi_{i=m}^{\infty}I_{i}$ (for $m$ a positive integer). Suppose $G$ is a closed subset of $I_{0} \times I_{1}$. We can define the \textit{topological entropy} of $G$ as follows:

\begin{enumerate}

\item First, let $\alpha=\{\alpha_{1}, \dots, \alpha_{n}\}$ be a minimal open cover of $I_{0}$ by open intervals.   Then $N^{*}(\alpha)=n$. For each positive integer $m>1$, let $$\alpha^{m}=\{ \Pi_{j=0}^{m-1} \alpha_{k_{j}} :  k_{j} \in \langle 1,n \rangle, 0 \le j \le m-1 \}. $$ 

\item If $K$ is a closed subset of $\Pi_{i=0}^{m-1}I_{i}$ ($m>1$ a positive integer or $m=\infty$), and $U$ is a collection of open sets in $\Pi_{i=0}^{m-1}I_{i}$ that covers $K$, let $N(K,U)$ denote the least cardinality of a subcover of $K$ in $U$. 

\item Then $\alpha^{2}=\{ \alpha_{i} \times \alpha_{j}:  1 \le i,j \le n \}$ is a cover of $G$ by open subsets of $I_{0} \times I_{1}$, and $N(G,\alpha^{2}) \le n^{2}$.   
\item   Now
$$ \alpha^{3}=\{ \alpha_{i_{0}} \times \alpha_{i_{1}} \times \alpha_{i_{2}}:  i_{k} \in \langle 1,n \rangle, 0 \le k \le 2\}, $$ is a cover of $G \star G$ by open sets in $\Pi_{j=0}^{2}I_{j}$ and $N({G\star G},\alpha^{3}) \le n^{3}$.

\item  Note that $\alpha^{2} \star \alpha^{2}$ contains more sets than does $\alpha^{3}$ since it contains sets of the form $ (\alpha_{i} \times \alpha_{j}) \star ( \alpha_{k} \times \alpha_{l})$ for $i,j,k,l \le n$, and $ (\alpha_{i} \times \alpha_{j}) \star ( \alpha_{k} \times \alpha_{l})=\alpha_{i} \times (\alpha_{j} \cap \alpha_{k} )\times \alpha_{l})$, which is nonempty as long as $ \alpha_{j} \cap  \alpha_{k} \ne \emptyset$ . However, a minimal subcover of $G \star G$ in  $\alpha^{2} \star \alpha^{2}$ has the same number of elements as a minimal subcover of  $\alpha^{3}$, since each set $\alpha_{i} \times (\alpha_{j} \cap \alpha_{k} )\times \alpha_{l})$ is contained at least one member of $\alpha^{3}$.

\item We can continue this process for each $m \in \mathbb{N}$:
$$ \alpha^{m+1}=\{ \Pi_{j=0}^{m} \alpha_{k_{j}} :  k_{j} \in \langle 1,n \rangle, 0 \le j \le m \} $$ is an open cover of $\star_{i=1}^{m}G$ and $N(\star_{i=1}^{m}G,\alpha^{m+1}) \le n^{m+1}$. Again, a minimal subcover of $ \star_{i=1}^{m}G$ by elements of $ \star_{i=1}^{m} \alpha^{2}$ has the same number of elements as a minimal subcover of $ \star_{i=1}^{m}G$ by elements of $ \alpha^{m+1}$. Since using the cover $ \star_{i=1}^{m} \alpha^{2}$ is sometimes more convenient, we continue to use both covers. Without loss of generality, we assume that a minimal subcover (in both $\alpha^{m+1}$ and $ \star_{i=1}^{m} \alpha^{2}$) consists of sets of the form $ \Pi_{j=0}^{m} \alpha_{k_{j}}$.

\item   Suppose $G$ is a closed subset of $I_{0} \times I_{1}$. Let $\mathbf{G}=\star_{i=1}^{\infty}G$.

    \begin{itemize}
     \item For each positive integer $m$, $0 \le N(\star_{i=1}^{m}G,\alpha^{m+1}) \le n^{m+1}$. If      $\mathbf{G} \ne \emptyset$,  $0 < N(\star_{i=1}^{m}G,\alpha^{m+1})$.

     \item If      $\mathbf{G} \ne \emptyset$,  $1= N(\star_{i=1}^{m}G,\alpha^{m+1})$ if and only if there is a sequence $\alpha_{j_{0}}, \alpha_{j_{1}}, \ldots \alpha_{j_{m}}$ (with each $1 \le j_{i}\le n$) such that $\mathbf{G} \subset (\alpha_{j_{0}} \times \ldots \times \alpha_{j_{m}}) \times I^{\infty}$.

   \item If $\alpha, \beta$ are both minimal open covers of $I_{0}$ by open intervals and $\alpha < \beta$, then for each $m>0$, $N(\star_{i=1}^{m}G,\alpha^{m+1}) \le N(\star_{i=1}^{m}G,\beta^{m+1})$.


\begin{proof} Let $k= N(\star_{i=1}^{m}G,\beta^{m+1})$.  Let $\{B_{1},B_{2}, \dots, B_{k} \}$ be a subcover of $\star_{i=1}^{m}G$ in $\beta^{m+1}$ of minimal cardinality. For each $1 \le i \le k$, there is some $A_{i} \in \alpha^{m+1}$ such that $B_{i} \subset A_{i}$. Then $\{A_{1},A_{2}, \dots, A_{k} \}$ is a subcover of $\star_{i=1}^{m}G$ in $\alpha^{m+1}$ of cardinality $k$. Hence,  
$N(\star_{i=1}^{m}G,\alpha^{m+1}) \le N(\star_{i=1}^{m}G,\beta^{m+1})$.

\end{proof}

   \item  If $\alpha$ is a minimal open cover of $I_{0}$ by open intervals, $k,l$ are positive integers, and $K \subset  \star_{i=1}^{k} G$, $L \subset \star_{i=1}^{l} G$, $K,L$ are closed, then $\alpha^{k+1} \star \alpha^{l}$ is a cover of $\Pi_{i=0}^{k+l} I_{i}$ and of $K \star L$ by open sets in $\Pi_{i=0}^{k+l} I_{i}$, as is $(\star_{i=1}^{k} \alpha^{2})\star (\star_{i=1}^{l} \alpha^{2})=\star_{i=1}^{k+l} \alpha^{2}$. Furthermore,  $N(\star_{i=1}^{k+l}G,\alpha^{k+l+1}) = N((\star_{i=1}^{k+l}G,\star_{i=1}^{k+l}\alpha^{2}) \le N(\star_{i=1}^{k}G,\alpha^{k+1}) N(\star_{i=1}^{l}G,\alpha^{l+1})= \newline N(\star_{i=1}^{k}G,\star_{i=1}^{k}\alpha^{2}) N(\star_{i=1}^{l}G,\star_{i=1}^{l}\alpha^{2})$.

\begin{proof} 

Showing that $\alpha^{k+1} \star \alpha^{l+1}$ and $(\star_{i=1}^{k} \alpha^{2})\star (\star_{i=1}^{l} \alpha^{2})=\star_{i=1}^{k+l} \alpha^{2}$ are open  covers of $\Pi_{i=0}^{k+l} I_{i}$ and of $K \star L$ by open sets in $\Pi_{i=0}^{k+l} I_{i}$ is straightforward and we omit it.  Let  $\{A_{1},A_{2}, \dots, A_{p} \}$ be a subcover of $\star_{i=1}^{k}G$ in $\alpha^{k+1}$ of minimal cardinality and let $\{B_{1},B_{2}, \dots, B_{q} \}$ be a subcover of $\star_{i=1}^{l}G$ in $\alpha^{l+1}$ of minimal cardinality. Then $\{A_{i} \star B_{j}: 1 \le i \le p, 1 \le j \le q \}$ is a subcover of  $\star_{i=1}^{k+l}G$ in $\alpha^{k+1} \star \alpha^{l+1}$, and

\begin{align*}
	N(\star_{i=1}^{k+l}G,\alpha^{k+1} \star \alpha^{l+1})= &N(\star_{i=1}^{k+l}G, \alpha^{k+l+1}) = N(\star_{i=1}^{k+l}G, \star_{i=1}^{k+l}\alpha^{2}) \\ 
	\le & N( \star_{i=1}^{k}G,\alpha^{k+1}) N(\star_{i=1}^{l}G,\alpha^{l+1}) \\
	= &N(\star_{i=1}^{k}G, \star_{i=1}^{k}\alpha^{2}) N(\star_{i=1}^{l}G, \star_{i=1}^{l}\alpha^{2}).
\end{align*}

\end{proof}

 \item  If $\alpha, \beta$ are both minimal open covers of $I_{0}$ by open intervals, then for each $m>0$, $\alpha^{m+1} \vee \beta^{m+1} = (\alpha \vee \beta)^{m+1}$, and $N(\star_{i=1}^{m}G,\alpha^{m+1} \vee \beta^{m+1}) \le N(\star_{i=1}^{m}G,\alpha^{m+1}) N(\star_{i=1}^{m}G,\beta^{m+1})$.

\begin{proof} 

Showing that $(\alpha \vee \beta)^{m+1} = \alpha^{m+1} \vee \beta^{m+1}$ is straightforward and we omit it.  Let  $\{A_{1},A_{2}, \dots, A_{k} \}$ be a subcover of $\star_{i=1}^{m}G$ in $\alpha^{m+1}$ of minimal cardinality and let $\{B_{1},B_{2}, \dots, B_{l} \}$ be a subcover of $\star_{i=1}^{m}G$ in $\beta^{m+1}$ of minimal cardinality. Then $\{A_{i} \cap B_{j}: 1 \le i \le k, 1 \le j \le l \}$ is a subcover of  $\star_{i=1}^{m}G$ in $\alpha^{m+1} \vee \beta^{m+1}$, and $N(\star_{i=1}^{m}G, \alpha^{m+1} \vee \beta^{m+1}) \le N( \star_{i=1}^{m}G,\alpha^{m+1}) N(\star_{i=1}^{m}G,\beta^{m+1})$.

\end{proof}

\item If $K$ is a closed subset of $G \subset I_{0} \times I_{1}$, $m$ is a positive integer, and $\alpha =\{\alpha_{1}, \dots, \alpha_{n}\}$ is a minimal open cover of $I_{0}$ by open intervals, then  $N(\star_{i=1}^{m}K,\alpha^{m+1}) \le N(\star_{i=1}^{m}G,\alpha^{m+1})$.

\begin{proof} Suppose $\{A_{1},A_{2}, \ldots, A_{k} \}$ is an open subcover of minimum cardinality of $\star_{i=1}^{m}G$ in $\alpha^{m+1}$. Since $\star_{i=1}^{m}K \subset \star_{i=1}^{m}G$, $\{A_{1},A_{2}, \ldots, A_{k} \}$ is also an open subcover of $K$. Hence, $N(\star_{i=1}^{m}K,\alpha^{m+1}) \le N(\star_{i=1}^{m}G,\alpha^{m+1})$.

\end{proof}

\item Suppose $l$ and $m$ are positive integers. Then $\alpha^{l+1}$ is a grid cover of $\star_{i=1}^{l}G$, $\alpha^{l+m+1}$ is a grid cover of $\star_{i=1}^{l+m}G$, and $\alpha^{l+1} \times \Pi_{i=l+1}^{m+1}I_{i}$ is an open cover of $\star_{i=1}^{l+m}G$.  Then $N(\star_{i=1}^{l}G,\alpha^{l+1}) \le N(\star_{i=1}^{l+m}G,\alpha^{l+1} \times \Pi_{i=l+1}^{l+m} I_i) \le N(\star_{i=1}^{l+m}G,\alpha^{l+m+1})$.


\begin{proof} 
Suppose $\{B_{j} \times \Pi_{i=l+1}^{l+m} I_i \}_{j=1}^{k} $ is a subcover of  $\star_{i=1}^{l+m}G$ in $\alpha^{l+1} \times \Pi_{i=l+1}^{l+m} I_i$ of least cardinality. Then $\{B_{j} \}_{j=1}^{k}$ is a subcover of  $\star_{i=1}^{l}G$ in $\alpha^{l+1}$ of least cardinality. Hence, $N(\star_{i=1}^{l}G,\alpha^{l+1} ) \le N(\star_{i=1}^{l+m}G,\alpha^{l+1}\times \Pi_{i=l+1}^{l+m} I_i )$. Since $\alpha^{l+m+1}$ refines $\alpha^{l+1}\times \Pi_{i=l+1}^{l+m} I_i$,   $N(\star_{i=1}^{l+m}G,\alpha^{l+1}\times \Pi_{i=l+1}^{l+m} I_i) \le  N(\star_{i=1}^{l+m}G,\alpha^{l+m+1}) $. The result follows.

\end{proof}
 \end{itemize}

\item 
If $\alpha=\{\alpha_{1}, \dots, \alpha_{n}\}$ is a minimum open cover of $I_{0}$ by intervals, $G$ is a closed subset of $I_{0} \times I_{1}$ and $\mathbf{G} \ne \emptyset$, then $ \underset{m \to \infty}{\lim} \frac{\log N(\star_{i=1}^{m}G,\alpha^{m+1})}{m} = \underset{m \to \infty}{\lim} \frac{\log N(\star_{i=1}^{m}G,\star_{i=1}^{m}\alpha^{2})}{m} $ exists.

\begin{proof}

Let $a_{m}= \log N(\star_{i=1}^{m}G,\alpha^{m+1})=\log N(\star_{i=1}^{m}G,\star_{i=1}^{m}\alpha^{2})$ for each $m \in \mathbb{N}$. Then $1 \le   N(\star_{i=1}^{m}G,\alpha^{m+1}) \le n^{m+1}$, so $$ 0 \le a_{m} =  \log N(\star_{i=1}^{m}G,\alpha^{m+1}) \le (m+1) \log n.$$ By Lemma \ref{Walters}, it suffices to show that $a_{m+k} \le a_{m} +a_{k}$.  We have $$\alpha^{m+k+1} \subset  (\star_{i=1}^{m} \alpha^{2}) \star (\star_{i=m+1}^{m+k} \alpha^{2}), $$ and $$N(\star_{i=1}^{m+k}G,\alpha^{m+k+1}) = N( \star_{i=1}^{m+k}G,(\star_{i=1}^{m} \alpha^{2}) \star (\star_{i=m+1}^{m+k} \alpha^{2})).$$

Since $N(\star_{i=1}^{m}G,\alpha^{m+1})$ is the cardinality of a minimal subcover of $\star_{i=1}^{m} G$  in $\star_{i=1}^{m} \alpha^{2}$, and $N(\star_{i=1}^{k}G,\alpha^{k+1})$ is the cardinality of a minimal subcover of $\star_{i=1}^{k} G=\star_{i=m+1}^{m+k} G$  in $\star_{i=m+1}^{m+k} \alpha^{2}$, $(\star_{i=1}^{m} \alpha^{2}) \star (\star_{i=m+1}^{m+k} \alpha^{2})$ is a cover of $\star_{i=1}^{m+k}G$  in $\prod_{i=0}^{m+k}I_{i}$. Thus, 

\begin{align*}
&N(\star_{i=1}^{m+k}G,\alpha^{m+k+1}) =  N(\star_{i=1}^{m+k}G, (\star_{i=1}^{m} \alpha^{2}) \star (\star_{i=m+1}^{m+k} \alpha^{2}) )  = \\
&  N( \star_{i=1}^{m+k}G,(\star_{i=1}^{m} \alpha^{2}) \star (\star_{i=1}^{k} \alpha^{2}) ) \le  N(\star_{i=1}^{m}G,\star_{i=1}^{m} \alpha^{2})  N( \star_{i=1}^{k}G,\star_{i=1}^{k} \alpha^{2})  ,
\end{align*}

and we have 

\begin{align*}
a_{m+k} &= \log(N(\star_{i=1}^{m+k}G,\alpha^{m+k+1}) ) \le \log(  N(\star_{i=1}^{m}G,\star_{i=1}^{m} \alpha^{2})  N(\star_{i=1}^{k}G, \star_{i=1}^{k} \alpha^{2}) ) \\
&=\log(  N(\star_{i=1}^{m}G,\star_{i=1}^{m} \alpha^{2})+  \log (N( \star_{i=1}^{k}G,\star_{i=1}^{k} \alpha^{2}) )=a_{m}+a_{k} .
\end{align*}

\end{proof}

\item If $\mathbf{G} \ne \emptyset$, define $\ent(G,\alpha)$ to be 

  \[ 
\ent(G,\alpha) = \lim_{m \to \infty} \frac{\log N(\star_{i=1}^{m}G,\alpha^{m+1})}{m}. 
\]

If $\mathbf{G} = \emptyset$, define $\ent(G,\alpha)=0$.

 
\item \begin{prop} \label{ent subset}
If $\alpha=\{\alpha_{1}, \dots, \alpha_{n}\}$ is a minimum open cover of $I_{0}$ by intervals, $G$ is a closed subset of $I_{0} \times I_{1},$ then we have the following:
\begin{enumerate}
	
	\item $\ent(G,\alpha) \ge 0$.
	
	\item If $\alpha < \beta$, $\alpha, \beta$ both minimal covers of $I_{0}$ by open intervals, then $\ent(G,\alpha) \le \ent(G,\beta)$.
	
	\item If $K$ is a closed subset of $G \subset I_{0} \times I_{1}$, $m$ is a positive integer, and $\alpha =\{\alpha_{1}, \dots, \alpha_{n}\}$ is a minimal open cover of $I_{0}$ by open intervals, then $\ent(K,\alpha) \le \ent(G,\alpha)$.
\end{enumerate}
\end{prop}
\begin{proof}
	(a) follows directly from the definition of $\ent(G,\alpha).$
	
	Let us prove (b):
	For each positive integer $m$, $\alpha^{m} < \beta^{m}$. If $\{B_{1}, \ldots, B_{k} \}$ is a minimal subcover of $\star_{i=1}^{m}G$ in $\beta^{m+1}$, then for each $1 \le i \le k$, there is some $A_{i} \in \alpha^{m+1}$ such that $B_{i} \subset A_{i}$. Thus, $\{A_{1}, \ldots, A_{k} \}$ is a subcover of $\star_{i=1}^{m}G$ in $\alpha^{m+1}$, and $\ent(G,\alpha) \le \ent(G, \beta)$.
	
	(c) follows directly from the fact that  $N(\star_{i=1}^{m}K,\alpha^{m+1}) \le \newline N(\star_{i=1}^{m}G,\alpha^{m+1})$ for each positive integer $m$.
\end{proof}

\item Finally, we define $\ent(G)= \underset{\alpha} {\sup} \{\ent(G, \alpha)\}$, where $\alpha$ ranges over all minimal covers of $I_{0}$ by open intervals (in $I_{0})$. 

\end{enumerate}

\begin{theorem} \label{G inverse ent}
Let $G$   be a closed subset of $I_{0} \times I_{1}$ and $G^{-1}=\{(x,y): (y,x) \in G \}$. Then $\ent(G)=\ent(G^{-1})$.

\end{theorem}

\begin{proof} 

For a positive integer $m$, note that a point $(x_{0},x_{1}, \ldots, x_{m}) \in \star_{i=1}^{m}G$ if and only if $(x_{m},x_{m-1}, \ldots, x_{0}) \in \star_{i=1}^{m}G^{-1}$.  Suppose $\alpha=\{\alpha_{1}, \dots, \alpha_{n} \}$ is a minimal cover of $I_{0}$ by open intervals. Suppose $m$ is a positive integer. Then $\alpha_{i_{0}} \times \alpha_{i_{1}} \times \cdots \times \alpha_{i_{m}} \in \alpha^{m+1}$ if and only if $\alpha_{i_{m}} \times \alpha_{i_{m-1}} \times \cdots \times \alpha_{i_{0}} \in \alpha^{m+1}$, and $(\alpha_{i_{0}} \times \alpha_{i_{1}} \times \cdots \times \alpha_{i_{m}}) \cap \star_{j=1}^{m}G \ne \emptyset$ if and only if $(\alpha_{i_{m}} \times \alpha_{i_{m-1}} \times \cdots \times \alpha_{i_{0}})\cap \star_{j=1}^{m}G^{-1} \ne \emptyset$. Then $N(\star_{i=1}^{m}G,\alpha^{m+1})=N(\star_{i=1}^{m}G^{-1},\alpha^{m+1})$ for each $m$. Hence, $\ent(G,\alpha)=\ent(G^{-1},\alpha)$ for each cover $\alpha$, and the result follows.

\end{proof}


\noindent \textbf{Remark} Theorem \ref{G inverse ent} above overlaps with Corollary 3.6 of \cite{KT}.

\begin{prop} \label{inverse cover}

 If $\mathcal{V}$ is an open cover (in $I^{\infty}$) of $\star^{\infty}_{i=1}G$, then $\sigma^{-1}(\mathcal{V}):=\{\sigma^{-1}(v):v \in \mathcal{V}\}=\{I_{0} \times v:v \in \mathcal{V}\}$ is also an open cover (in $I^{\infty}$) of  $\star^{\infty}_{i=1}G$. 
 \end{prop}

\begin{proof}
Suppose $x=(x_{0},x_{1}, \ldots) \in   \star^{\infty}_{i=1}G$. Then $\sigma(x)=(x_{1},x_{2}, \ldots) \in \star^{\infty}_{i=1}G$, and  there is some $v \in \mathcal{V}$ such that $\sigma(x) \in v$. Since $I_{0} \times v = \sigma^{-1}(v)$, $x \in \sigma^{-1}(v) \in \sigma^{-1}(\mathcal{V})$. 

\end{proof}

\vphantom{}

If $\mathcal{U}$ is an open cover of $I^{\infty}$, $G$ is a closed subset of $I_{0} \times I_{1}$, and $\mathbf{G}=\star_{i=1}^{\infty}G \ne \emptyset$, let $\mathcal{U}^{*}= \{ u \cap \mathbf{G}: u \in \mathcal{U} \}$ denote the corresponding open cover of $\mathbf{G}$ by open sets in $\mathbf{G}$.

\begin{theorem} \label{sigma versus G}
Suppose $G$ is a closed subset of $I_{0} \times I_{1}$, $\mathbf{G} =\star_{i=1}^{\infty}G\ne \emptyset$,  and $\sigma(\mathbf{G})=\mathbf{G}$.  If $\alpha=\{\alpha_{1}, \dots, \alpha_{n}\}$ is a minimal open cover of $I_{0}$ by open  intervals, then $\ent(G,\alpha)=h(\sigma,(\alpha^{M+1} \times I^{\infty})^{*})$ for each positive integer $M$.

\end{theorem}

\begin{proof}

   Let $\alpha=\{\alpha_{1}, \dots, \alpha_{n}\}$ be a minimal open cover of $I_{0}$ by intervals.

Fix the positive integer $M$. Let $$\mathcal{V}=\left \{ \prod_{j=0}^{M} \alpha_{i_{j}} \times I^{\infty}: \alpha_{i_{j}} \in \alpha \text{ and } \left(\prod_{j=0}^{M} \alpha_{i_{j}} \times I^{\infty}\right) \cap \mathbf{G} \ne \emptyset \right \},$$ 
and let 
\begin{align*}
\mathcal{U}= &\left \{ \alpha_{k_{0}} \times \alpha_{k_{1}} \times \ldots \times \alpha_{k_{M}} \times \alpha_{k_{M+1}} \times I^{\infty}:  \{k_{j}\}_{j=0}^{M+1} \text{ is a sequence of} \right.\\
& \left.  \text{ members of } \{1, \ldots, n \} \text{ of length } M+2 \right. \big \}.
\end{align*}


For $v=\prod_{j=0}^{M} \alpha_{i_{j}} \times I^{\infty} \in \mathcal{V}$, $\sigma^{-1}(v)=I_{0} \times \prod_{j=0}^{M} \alpha_{i_{j}} \times I^{\infty}$. Then
 \begin{align*}
 \sigma^{-1}(\mathcal{V}) \vee \mathcal{V}& =  \left \{ \sigma^{-1}(v) \cap w: v=\prod_{j=0}^{M} \alpha_{i_{j}} \times I^{\infty}, w=\prod_{j=0}^{M} \alpha_{k_{j}} \times I^{\infty} \in \mathcal{V}  \right \} \\
 &= \left \{ \left(I_{0} \times \prod_{j=0}^{M} \alpha_{i_{j}} \times I^{\infty}\right) \cap \left(\prod_{j=0}^{M} \alpha_{k_{j}} \times I^{\infty}\right): \{i_{j}\}_{j=0}^{M} \text{ and } \{k_{j} \}_{j=0}^{M} \text{ are } \right. \\
 &\left.  \text{ finite sequences of members of }  \{1,\ldots,n \} \newline \text{ of length } M+1  \rule{0cm}{0.7cm} \right \} .
\end{align*}
\vphantom{}
 
If $v=\prod_{j=0}^{M} \alpha_{i_{j}} \times I^{\infty} \in \mathcal{V}$ and $w=\prod_{j=0}^{M} \alpha_{k_{j}} \times I^{\infty} \in \mathcal{V}$, then 

\begin{align*}
 \sigma^{-1}(v) \cap w &=\left( I_{0} \times \prod_{j=0}^{M} \alpha_{i_{j}} \times I^{\infty}\right) \cap \left(\prod_{j=0}^{M} \alpha_{k_{j}} \times I^{\infty}\right)\\
 &=\alpha_{k_{0}} \times \left(\alpha_{k_{1}} \cap \alpha_{i_{0}}\right) \times \ldots \times \left(\alpha_{k_{M}} \cap \alpha_{i_{M-1}}\right) \times \alpha_{i_{M}} \times I^{\infty} \\
 & \subset \alpha_{k_{0}} \times \alpha_{k_{1}} \times \ldots \times \alpha_{k_{M}} \times \alpha_{i_{M}} \times I^{\infty}.
 \end{align*}


Hence, the collection $\sigma^{-1}(\mathcal{V}) \vee \mathcal{V}$ refines the collection $\mathcal U.$
 Then $(\sigma^{-1}(\mathcal{V}) \vee \mathcal{V})^{*}$ refines the collection $ \mathcal{U}^{*}$, so $\mathcal{U}^{*} < (\sigma^{-1}(\mathcal{V}) \vee \mathcal{V})^{*}$ , and $N(\mathbf{G},\mathcal{U}^{*}) \le N(\mathbf{G}, (\sigma^{-1}(\mathcal{V}) \vee \mathcal{V})^{*}). $

\vphantom{}

But $ \mathcal{U}$ also refines  $\sigma^{-1}(\mathcal{V}) \vee \mathcal{V}$, and so $ \mathcal{U}^{*}$  refines  $(\sigma^{-1}(\mathcal{V}) \vee \mathcal{V})^{*}$. Thus, $N(\mathbf{G},\mathcal{U}^{*}) \ge N(\mathbf{G}, (\sigma^{-1}(\mathcal{V}) \vee \mathcal{V})^{*})$. Then $N(\mathbf{G},\mathcal{U}^{*}) = N( \mathbf{G},(\sigma^{-1}(\mathcal{V}) \vee \mathcal{V})^{*})$.

\vphantom{}

Note that $N(\mathbf{G},\mathcal{U}^{*}) =N(\star_{i=1}^{M+1}G,\alpha^{M+2})$.

\vphantom{}

  We can continue: By similar arguments, for each positive integer $l$, \\ $N(\mathbf{G},(\vee_{i=0}^{l} \sigma^{-i}\mathcal{V})^{*})=N(\mathbf{G},\alpha^{M+l+1} \times I^{\infty})=N(\star_{i=1}^{M+l+1}G,\alpha^{M+l+1})$. Now $\mathcal{V} = \alpha^{M+1} \times I^{\infty}$, and for $l$ a positive integer, $N(\mathbf{G},(\vee_{i=0}^{l} \sigma^{-i}\mathcal{V})^{*})=N(\star_{i=1}^{M+l}G,\alpha^{M+l+1})$. Then  $\log(N(\mathbf{G},(\vee_{i=0}^{l} \sigma^{-i}\mathcal{V})^{*}))=\log(N(\star_{i=1}^{M+l}G,\alpha^{M+l+1})).$ It follows that 

\begin{align*}
h(\sigma, \alpha^{M+1} \times I^{\infty})&= \lim_{l \to \infty} \frac{\log(N(\mathbf{G},(\vee_{i=0}^{l} \sigma^{-i}\mathcal{V})^{*}))}{l}
 = \lim_{l \to \infty} \frac{\log(N(\star_{i=1}^{M+l}G,\alpha^{M+l+1}))}{l},
\end{align*}
while $$\ent(G,\alpha) = \lim_{l \to \infty}\frac{\log(N(\star_{i=1}^{l}G,\alpha^{l+1}))}{l}.$$ 


\vphantom{}

For each positive integer $k$, let $\log(N(\star_{i=1}^{k}G,\alpha^{k+1})) = a_{k}$.\\ Then $\log(N(\star_{i=1}^{l}G,\alpha^{l+1})) = a_{l}$ and $\log(N(\star_{i=1}^{M+l}G,\alpha^{M+l+1})) = a_{M+l}$. Furthermore, $a_{l} \le a_{M+l} \le a_{M}+a_{l}$. (This is because  $N(\star_{i=1}^{M+l}G,\alpha^{M+l+1}) \le \newline  N(\star_{i=1}^{M}G,\alpha^{M+1})N(\star_{i=1}^{l}G,\alpha^{l+1})  $.) Then $\frac{a_{l}}{l} \le \frac{a_{l+M}}{l} \le \frac{a_{l}}{l}+ \frac{a_{M}}{l}$. By Lemma \ref{ent lim}, $\lim_{l \to \infty} \frac{a_{l}}{l}$ exists, and 
$$
\lim_{l \to \infty} \frac{a_{l}}{l} \le  \lim_{l \to \infty}\frac{a_{l+M}}{l} \le \lim_{l \to \infty} (\frac{a_{l}}{l}+\frac{a_{M}}{l})= \lim_{l \to \infty} \frac{a_{l}}{l}+ 
\lim_{l \to {\infty}}\frac{a_{M}}{l} = \lim_{l \to {\infty}}\frac{a_{l}}{l}.
$$
It follows that  $$\lim_{l \to \infty}\frac{\log(N(\star_{i=1}^{l}G,\alpha^{l+1}))}{l}= \lim_{l \to \infty}\frac{\log(N(\star_{i=1}^{M+l}G,\alpha^{M+l+1}))}{l},$$ and thus, $\ent(G,\alpha)=h(\sigma, (\alpha^{M+1}\times I^{\infty})^{*})$ for each positive integer $M$.

\end{proof}

\begin{theorem} \label{ent and sigma}
Suppose $G$ is a closed subset of $I_{0} \times I_{1}$, $\mathbf{G} =\star_{i=1}^{\infty}G$, and $\sigma(\mathbf{G})=\mathbf{G}$.  If $\alpha=\{\alpha_{1}, \dots, \alpha_{n}\}$ is a minimal open cover of $I_{0}$ by open intervals, then $\ent(G)=h(\sigma)$.

\end{theorem}

\begin{proof}
 Since each open cover of $\mathbf{G}$ is refined by the grid cover $\alpha^{M+1} \times I^{\infty}$ for some $M$ and minimal open cover $\alpha$ by intervals of $I_{0}$, the result follows.

\end{proof}

\noindent \textbf{Remark} Theorem \ref{ent and sigma} overlaps with Theorem 3.1 of \cite{KT}.

\begin{theorem} \label{f continuous} 

If $f:I \to I$ is a continuous function and $G$ is the graph of $f^{-1}$, then $h(f)=\ent(G)$.

\end{theorem}

\begin{proof} This follows from Theorem \ref{sigma versus G}, and Bowen's result that $h(f)=h(\sigma)$ in \cite{Bo2}.

\end{proof}


\section{Topological entropy of closed subsets of $[0,1]^{N+1}$}

If $X$ is a compact metric space and $f:X\to X$ is continuous, then for each positive integer $k$, $h(f^k)=kh(f)$ (See \cite[Theorem 7.10]{Walters book}).  This well-known result for continuous mappings does \underline{not} hold for an upper semicontinuous mappings $F:X\to 2^X$.

Suppose $F:X\to 2^X$ is upper semicontinuous. We can define $F^2:X\to 2^X$ by $F^2(x)=\bigcup_{y\in F(x)} F(y).$ Then, inductively, for $n>2$, $F^n(x)=\bigcup_{y\in F^{n-1}(x)} F(y).$
In \cite{KT}, using Bowen's ideas of $(n,\epsilon)$-separated and $(n,\epsilon)$-spanning, they define the topological entropy $h(F).$ They show that, $h(\sigma)=h(F)$ (where $\sigma$ denotes the shift $\sigma: \varprojlim (X,f) \to \varprojlim (X,f)$).
Hence if $X=[0,1],$ $G$ denotes the graph of $F$ (and thus $G$ is a closed subset of $[0,1]^2$), $h(F)=h(\sigma)=\ent(G^{-1})=\ent(G).$ (See Theorem \ref{G inverse ent} and Theorem \ref{ent and sigma} of the previous section). 
In \cite{KT}, the folllowing theorem (re-phrased a bit for our setting) is proved:

\begin{theorem}\cite[Theorem 5.4]{KT}
	Suppose $X$ is a compact metric space, $F:X\to 2^X$ is upper semicontinuous, and $k\in \mathbb N.$ Then 
	$$
	h(F)\le h(F^k)\le k h(F).
	$$
\end{theorem}
It is not the case $h(F^k)= k h(F)$ always as the following example (from \cite{KT}) shows.

\begin{exam} \cite[Example 5.7]{KT} \label{ExamKT}
	Define $F:I\to 2^I$ by 
	$$
	F(x)=\left\{ 
	\begin{array}{ll}
	\{ x+\frac{1}{2}, \frac{1}{2}-x\}, & x\le \frac{1}{2} \\
	\{ x-\frac{1}{2}, \frac{3}{2}-x\}, & x\ge \frac{1}{2}
	\end{array}
	\right.
	$$
	Then $F^2\neq F,$ but $h(F^2)=h(F)=\log 2.$
\end{exam}

	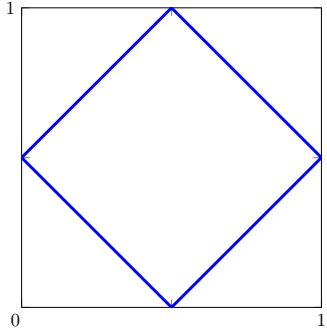
\begin{figure}[h]
	\centering
	\begin{tikzpicture}[scale=0.7]
	\begin{axis}[ 
	unit vector ratio*=1 1 1, 
	xtick={0,0.5,1}, xticklabels={$0 \ \ $,,1},
	ytick={0,0.5,1}, yticklabels={,,1},
	xmin=0,     
	xmax=1.,   
	ymin=0,     
	ymax=1,
	] 
	\addplot [blue, ultra thick] coordinates {(0,0.5) (0.5,1) (1,0.5) (0.5,0) (0,0.5)};
	\end{axis}
	\end{tikzpicture}
	\caption{The graph of the function $F$ from Example \ref{ExamKT}}
\end{figure}

However, an analogous result to $$h(f^k)= k h(f)$$ holds for Mahavier products. Note that if $G=\Gamma(F^{-1})$, where $F:X\to 2^X$ is upper semicontinuous, then $\Gamma(F^k)=\pi_{\{k,0\}} (\star_{i=1}^k G)$. Thus, only the beginning input and final output are kept, while the ``intermediate'' values are forgotten, when iterating a function in the usual way. This does not happen with the Mahavier product: \textit{all} possibities are retained.

\vphantom{}

The purpose of this section is to show that if $G$ is a closed subset of $I^2,$ and $k$ is a positive integer, then
$$
\ent(\star_{i=1}^k G)=k \ent(G).
$$

 Before we can show this, we must define and explore $\ent(H)$ for $H$ a closed subset of $I^{N+1},$ for $N$ a positive integer.

\begin{enumerate}

\item Let $\alpha=\{\alpha_{1}, \dots, \alpha_{n}\}$ be a minimal open cover of $I_{0}$ by intervals.    Let 
$$\beta =\{ \Pi_{j=0}^{N} \alpha_{k_{j}} :  k_{j} \in \langle 1,n \rangle, 0 \le j \le N \}.$$

Hence, $\beta$ is the grid cover of $\Pi_{i=0}^{N} I_{i}$ determined by $\alpha$, and $\beta$ therefore covers $H$. Since $N^{*}(\beta)=n(N+1):=n_{\beta}$, we can list the members of $\beta=\{\beta_{1}, \beta_{2}, \ldots, \beta_{n_{\beta}} \}$. 
For each positive integer $m>1$, let $$\beta^{m}=\{ \Pi_{j=0}^{m-1} \beta_{k_{j}} :  k_{j} \in \langle 1,n_{\beta} \rangle, 0 \le j \le m-1 \}. $$ 

\item Then $\beta \star \beta=\{ \beta_{i} \star \beta_{j}:  1 \le i,j \le n_{\beta} \}$ is a cover of $H \star H$ by open subsets of $\Pi_{i=0}^{2N} I_{i}$, and $N(H \star H,\beta \star \beta) \le n_{\beta}^{2}$. Note that $\beta \star \beta$ refines $\alpha^{2N+1}$ and $\beta \star \beta$ is refined by $\alpha^{2N+1}$, so $N(H \star H, \beta \star \beta)=N(H \star H, \alpha^{2N+1})$.

\item   We can continue this process for each $m \in \mathbb{N}$:
$$ \star_{i=1}^{m}\beta=\{ \star_{j=1}^{m} \beta_{k_{j}} :  k_{j} \in \langle 1,n_{\beta} \rangle, 1 \le j \le m \} $$ is an open cover of $\star_{i=1}^{m}H$ and $N(\star_{i=1}^{m}H,\star_{i=1}^{m}\beta) \le n_{\beta}^{m}$. Again, a minimal subcover of $ \star_{i=1}^{m}H$ by elements of $ \star_{i=1}^{m} \beta$ has the same number of elements as a minimal subcover of $ \star_{i=1}^{m}H$ by elements of $ \alpha^{mN+1}$. Since using the cover $ \star_{i=1}^{m} \beta$ is sometimes more convenient, we continue to use both covers. Without loss of generality, we may assume that a minimal subcover (in both $\alpha^{mN+1}$ and $ \star_{i=1}^{m} \beta$) consists of sets of the form $ \Pi_{j=0}^{mN} \alpha_{k_{j}}$, where each $k_{j} \in \langle 1,n \rangle$.
  
  
\item     Suppose $H$ is a closed subset of $\Pi_{i=0}^{N}I_{i}$. Let $\mathbf{H}=\star_{i=1}^{\infty}H$.

    \begin{itemize}
     \item For each positive integer $m$, $0 \le N(\star_{i=1}^{m}H,\alpha^{mN+1})= \newline N(\star_{i=1}^{m}H,\star_{i=1}^{m}\beta)  \le n^{mN+1}$, and $0 \le N(\star_{i=1}^{m}H,\alpha^{mN+1})= \newline N(\star_{i=1}^{m}H,\star_{i=1}^{m}\beta) 
 \le n_{\beta}^m$. If      $\mathbf{H} \ne \emptyset$,  $0 < N(\star_{i=1}^{m}H, \star_{i=1}^{m}\beta)$.

     \item If      $\mathbf{H} \ne \emptyset$,  $1= N(\star_{i=1}^{m}H,\star_{i=1}^{m}\beta)$ if and only if there is a sequence $\alpha_{j_{0}}, \alpha_{j_{1}}, \ldots \alpha_{j_{mN}}$ (with each $1 \le j_{i}\le n$) such that $\mathbf{H} \subset (\alpha_{j_{0}} \times \ldots \times \alpha_{j_{mN}}) \times I^{\infty}$.

\item As before, if $\alpha, \gamma$ are both minimal open covers of $I_{0}$ by open intervals and $\alpha < \gamma$, then for each $m>0$, $N(\star_{i=1}^{m}H,\alpha^{mN+1}) \le N(\star_{i=1}^{m}H,\gamma^{mN+1})$. 

 \item  As before, if $\alpha, \gamma$ are both minimal open covers of $I_{0}$ by open intervals, then for each $m>0$, $\alpha^{mN+1} \vee \gamma^{mN+1} = (\alpha \vee \gamma)^{mN+1}$, and $N(\star_{i=1}^{m}H,\alpha^{m+1} \vee \gamma^{m+1}) \le N(\star_{i=1}^{m}H,\alpha^{mN+1}) N(\star_{i=1}^{m}H,\gamma^{mN+1})$.

\item If $K$ is a closed subset of $H \subset \Pi_{i=0}^{N}I_{i}$, $m$ is a positive integer, and $\alpha =\{\alpha_{1}, \dots, \alpha_{n}\}$ is a minimal open cover of $I_{0}$ by open intervals, then  $N(\star_{i=1}^{m}K,\alpha^{mN+1}) \le N(\star_{i=1}^{m}H,\alpha^{mN+1})$.

\item Suppose $l$ and $m$ are positive integers. Then $\alpha^{lN+1}$ is a grid cover of $\star_{i=1}^{l}H$, $\alpha^{lN+mN+1}$ is a grid cover of $\star_{i=1}^{l+m}H$, and $\alpha^{lN+1} \times \Pi_{i=lN+1}^{mN+1}I_{i}$ is an open cover of $\star_{i=1}^{l+m}H$.  Then $N(\star_{i=1}^{l}H,\alpha^{lN+1}) \le N(\star_{i=1}^{l+m}H,\alpha^{lN+1} \times \Pi_{i=lN+1}^{lN+mN}I_{i}) \le N(\star_{i=1}^{l+m}H,\alpha^{lN+mN+1})$.

 \end{itemize}

\item 
If $\alpha=\{\alpha_{1}, \dots, \alpha_{n}\}$ is a minimal open cover of $I_{0}$ by intervals, $H$ is a closed subset of $\Pi_{i=0}^{N}I_{i}$ and $\mathbf{H} \ne \emptyset$, then $ \underset{m \to \infty}{\lim} \frac{\log N(\star_{i=1}^{m}H,\alpha^{mN+1})}{m} = \underset{m \to \infty}{\lim} \frac{\log N(\star_{i=1}^{m}H,\star_{i=1}^{m}\beta)}{m} $ exists.

\begin{proof}

Let $a_{m}= \log N(\star_{i=1}^{m}H,\alpha^{mN+1})=\log N(\star_{i=1}^{m}H,\star_{i=1}^{m}\beta)$ for each $m \in \mathbb{N}$. Then $1 \le   N(\star_{i=1}^{m}H,\alpha^{mN+1}) \le n^{mN+1}$, so $$ 0 \le a_{m} =  \log N(\star_{i=1}^{m}H,\alpha^{mN+1}) \le (mN+1) \log n.$$ By Lemma \ref{Walters}, it suffices to show that $a_{m+k} \le a_{m} +a_{k}$.  We have $$\alpha^{m+k+1} \subset  (\star_{i=1}^{m} \beta) \star (\star_{i=1}^{k} \beta)= \star_{i=1}^{m+k} \beta, $$ so $\alpha^{m+k+1}$  refines  $(\star_{i=1}^{m} \beta) \star (\star_{i=1}^{k} \beta).$ 

Then $$N(\star_{i=1}^{m+k}H,\alpha^{mN+kN+1}) = N( \star_{i=1}^{m+k}H,(\star_{i=1}^{m} \beta) \star (\star_{i=1}^{k} \beta)).$$

Since $N(\star_{i=1}^{m}H,\alpha^{mN+1})$ is the cardinality of a minimal subcover of $\star_{i=1}^{m} H$  in $\star_{i=1}^{m} \beta$, and $N(\star_{i=1}^{k}H,\alpha^{kN+1})$ is the cardinality of a minimal subcover of $\star_{i=1}^{k} H$   in $\star_{i=1}^{k} \beta$, $(\star_{i=1}^{m} \beta) \star (\star_{i=m+1}^{m+k} \beta)$ is a cover of $\star_{i=1}^{m+k}H$  in $\Pi_{i=0}^{mN+kN}I_{i}$. \\

Thus, 
\begin{align*}
&N(\star_{i=1}^{m+k}H,\alpha^{mN+kN+1}) =  N(\star_{i=1}^{m+k}H, (\star_{i=1}^{m} \beta) \star (\star_{i=m+1}^{m+k} \beta) )  =  \\
& N( \star_{i=1}^{m+k}H,(\star_{i=1}^{m} \beta) \star (\star_{i=1}^{k} \beta) ) \le  N(\star_{i=1}^{m}H,\star_{i=1}^{m} \beta)  N( \star_{i=1}^{k}H,\star_{i=1}^{k} \beta),
\end{align*}

and we have 
\begin{align*}
a_{m+k} &= \log(N(\star_{i=1}^{m+k}H,\alpha^{mN+kN+1}) ) \le \log(  N(\star_{i=1}^{m}H,\star_{i=1}^{m} \beta)  N(\star_{i=1}^{k}H, \star_{i=1}^{k} \beta) )\\
&=\log(  N(\star_{i=1}^{m}H,\star_{i=1}^{m} \beta)+  \log (N( \star_{i=1}^{k}H,\star_{i=1}^{k} \beta) )=a_{m}+a_{k} .
\end{align*}
\end{proof}


\item If $\mathbf{H} \ne \emptyset$, define $\ent(H,\alpha)$ to be 

  \[ 
\ent(H,\alpha) = \lim_{m \to \infty} \frac{\log N(\star_{i=1}^{m}H,\alpha^{mN+1})}{m}. 
\]

If $\mathbf{H} = \emptyset$, define $\ent(H,\alpha)=0$.

\item We have the following:

\begin{enumerate}

\item $\ent(H,\alpha) \ge 0$.

\item If $\alpha < \beta$, $\alpha, \beta$ both minimal covers of $I_{0}$ by open intervals, then $\ent(H,\alpha) \le \ent(H,\beta)$.

\item If $K$ is a closed subset of $H \subset \Pi_{i=0}^{N}I_{i}$, $m$ is a positive integer, and $\alpha =\{\alpha_{1}, \dots, \alpha_{n}\}$ is a minimal open cover of $I_{0}$ by open intervals, then $\ent(K,\alpha) \le \ent(H,\alpha)$. 

\end{enumerate}

\item Finally, we define $\ent(H)= \underset{\alpha} {\sup} \{\ent(H, \alpha)\}$, where $\alpha$ ranges over all minimal covers of $I_{0}$ by open intervals (in $I_{0})$.

\end{enumerate}

\begin{theorem}
Let $H$   be a closed subset of $\Pi_{i=0}^{N}I_{i}$ and $H^{-1}=\{(x_{N},x_{N-1}, \dots, x_{1},x_{0}): (x_{0},x_{1}, \dots, x_{N-1},x_{N}) \in H \}$. Then $\ent(H)=\ent(H^{-1})$.

\end{theorem}

\begin{proof} The proof is similar to that of Theorem \ref{G inverse ent}, so we omit it.

\end{proof}

\begin{prop}

 If $\mathcal{V}$ is an open cover (in $I^{\infty}$) of $\star^{\infty}_{i=1}H$, with $H$ a closed subset of $ \Pi_{i=0}^{N} I_i$, then $\sigma^{-N}(\mathcal{V}):=\{\sigma^{-N}(v):v \in \mathcal{V}\}=\{\Pi_{i=0}^{N-1}I_{i} \times v:v \in \mathcal{V}\}$ is also an open cover (in $I^{\infty}$) of  $\star^{\infty}_{i=1}H$. 
 \end{prop}

\begin{proof} 
The proof is similar to that of Proposition \ref{inverse cover}, so we omit it.

\end{proof}

Suppose $H$ is a closed subset of $ \Pi_{i=0}^{N}I_{i}$ and $\mathbf{H}=\star_{i=1}^{\infty}H$. If $\mathcal{U}$ is an open cover of $I^{\infty}$, let $\mathcal{U}^{*}= \{ u \cap \mathbf{H}: u \in \mathcal{U} \}$ denote the corresponding open cover of $\mathbf{H}$ by open sets in $\mathbf{H}$.

\begin{theorem}
Suppose $H$ is a closed subset of  $ \Pi_{i=0}^{N}I_{i}$, $\mathbf{H} =\star_{i=1}^{\infty}H\ne \emptyset$,  and $\sigma^{N}(\mathbf{H})=\mathbf{H}$.  Suppose $M$ is a positive integer. If $\alpha=\{\alpha_{1}, \dots, \alpha_{n}\}$ is a minimal open cover of $I_{0}$ by open  intervals, then $\ent(H,\alpha)=h(\sigma^{N},(\alpha^{MN+1} \times I^{\infty})^{*})$.

\end{theorem}

\begin{proof} This proof is similar to the proof of Theorem \ref{sigma versus G}, but a little more difficult technically.

  Let $\alpha=\{\alpha_{1}, \dots, \alpha_{n}\}$ be a minimal open cover of $I_{0}$ by intervals. 

\vphantom{}

Fix the positive integer $M$. Let 
$$\mathcal{V}=\left \{ \prod_{j=0}^{MN} \alpha_{i_{j}} \times I^{\infty}: \alpha_{i_{j}} \in \alpha \text{ and } \left (\prod_{j=0}^{MN} \alpha_{i_{j}} \times I^{\infty} \right) \cap \mathbf{H} \ne \emptyset \right \},$$ 
and let 
\begin{align*}
\mathcal{U}=& \left \{ \alpha_{k_{0}} \times \alpha_{k_{1}} \times \ldots \times \alpha_{k_{(M+1)N}} \times I^{\infty}:  \{k_{j}\}_{j=0}^{(M+1)N} \text{ is a sequence of} \right. \\
& \text{ members of }\{1, \ldots, n \} \text{ of length } (M+1)N+1  \Big \}.
\end{align*}

For $v=\prod_{j=0}^{MN} \alpha_{i_{j}} \times I^{\infty} \in \mathcal{V}$, $\sigma^{-N}(v)=\prod_{i=0}^{N-1}I_{i} \times \prod_{j=0}^{MN} \alpha_{i_{j}} \times I^{\infty}$. Then 
\begin{align*}
\sigma^{-N}(\mathcal{V}) \vee \mathcal{V} & =\left \{ \sigma^{-N}(v) \cap w: v  =\prod_{j=0}^{MN} \alpha_{i_{j}} \times I^{\infty}, w=\prod_{j=0}^{MN} \alpha_{k_{j}} \times I^{\infty} \in \mathcal{V} \right \} \\
& = \left \{ \left(\prod_{i=0}^{N-1}I_{i} \times \prod_{j=0}^{MN} \alpha_{i_{j}} \times I^{\infty}\right) \cap \left(\prod_{j=0}^{MN} \alpha_{k_{j}} \times I^{\infty}\right): \{i_{j}\}_{j=0}^{MN} \text{ and } \{k_{j} \}_{j=0}^{MN} \right. \\
& \left. \text{ are finite sequences of} \newline \text{ members of }  \{1,\ldots,n \}  \text{ of length } (M+1)N+1 \rule{0cm}{0.7cm} \right \}.
\end{align*}

\vphantom{}
 
If $v=\prod_{j=0}^{MN} \alpha_{i_{j}} \times I^{\infty} \in \mathcal{V}$ and $w=\prod_{j=0}^{MN} \alpha_{k_{j}} \times I^{\infty} \in \mathcal{V}$, then 

\begin{align*}
 &\sigma^{-N}(v) \cap w =\left(\prod_{i=0}^{N-1}I_{i} \times \prod_{j=0}^{MN} \alpha_{i_{j}} \times I^{\infty}\right) \cap \left(\prod_{j=0}^{MN} \alpha_{k_{j}} \times I^{\infty}\right) \\
 &= \prod_{i=0}^{N-1}\alpha_{k_{i}} \times (\alpha_{k_{N}} \cap \alpha_{i_{0}}) \times \ldots \times (\alpha_{k_{MN}} \cap \alpha_{i_{(M-1)N}}) \times \prod_{l=N(M-1)+1}^{MN}\alpha_{i_{l}} \times I^{\infty} \\ 
 &\subset  \alpha_{k_{0}} \times \alpha_{k_{1}} \times \ldots \times \alpha_{k_{MN}} \times \alpha_{i_{(M-1)N+1}} \times \ldots \times \alpha_{i_{MN}} \times I^{\infty}.
 \end{align*}

Hence, the collection $\sigma^{-1}(\mathcal{V}) \vee \mathcal{V}$ refines the collection $ \mathcal{U}.$ Then $(\sigma^{-N}(\mathcal{V}) \vee \mathcal{V})^{*}$ refines the collection $ \mathcal{U}^{*}$, so $\mathcal{U}^{*} < (\sigma^{-N}(\mathcal{V}) \vee \mathcal{V})^{*}$ , and $N(\mathbf{H},\mathcal{U}^{*}) \le N(\mathbf{H}, (\sigma^{-N}(\mathcal{V}) \vee \mathcal{V})^{*})$. 

\vphantom{}

But $ \mathcal{U}$ also refines  $\sigma^{-N}(\mathcal{V}) \vee \mathcal{V}$, and so $ \mathcal{U}^{*}$  refines  $(\sigma^{-N}(\mathcal{V}) \vee \mathcal{V})^{*}$. Thus, $N(\mathbf{H},\mathcal{U}^{*}) \ge N(\mathbf{H}, (\sigma^{-N}(\mathcal{V}) \vee \mathcal{V})^{*})$. Then $N(\mathbf{H},\mathcal{U}^{*}) = N( \mathbf{H},(\sigma^{-N}(\mathcal{V}) \vee \mathcal{V})^{*})$.

\vphantom{}

Note that $N(\mathbf{H},\mathcal{U}^{*}) =N(\star_{i=1}^{(M+1)N}H,\alpha^{MN+1})$.

\vphantom{}

  We can continue: By similar arguments, for each positive integer $l$, \newline $N(\mathbf{H},(\vee_{i=0}^{l} \sigma^{-iN}(\mathcal{V})^{*}))=N(\mathbf{H},\alpha^{(M+l)N+1} \times I^{\infty})=N(\star_{i=1}^{(M+l)N}H,\alpha^{(M+l)N+1})$. Now $\mathcal{V} = \alpha^{MN+1} \times I^{\infty}$, and for $l$ a positive integer, $N(\mathbf{H},(\vee_{i=0}^{l} \sigma^{-iN}(\mathcal{V})^{*}))=N(\star_{i=1}^{M+l}H,\alpha^{(M+l)N+1})$. Then  $\log(N(\mathbf{H},(\vee_{i=0}^{l} \sigma^{-iN}(\mathcal{V})^{*}))= \newline \log(N(\star_{i=1}^{(M+l)N}H,\alpha^{(M+l)N+1})).$ It follows that 
\begin{align*}
h(\sigma^{N}, \alpha^{MN+1} \times I^{\infty}) &= \lim_{l \to \infty} \frac{\log(N(\mathbf{H},(\vee_{i=0}^{l} \sigma^{-iN}(\mathcal{V})^{*}))}{l} \\
& = \lim_{l \to \infty} \frac{\log(N(\star_{i=1}^{(M+l)N}H,\alpha^{(M+l)N+1}))}{l},
\end{align*}
while
$$\ent(H,\alpha) = \lim_{l \to \infty}\frac{\log(N(\star_{i=1}^{l}H,\alpha^{lN+1}))}{l}.$$

\vphantom{}

For each positive integer $k$, let $\log(N(\star_{i=1}^{k}H,\alpha^{kN+1})) = a_{k}$. Then \newline $\log(N(\star_{i=1}^{l}H,\alpha^{lN+1})) = a_{l}$ and $\log(N(\star_{i=1}^{M+l}H,\alpha^{(M+l)N+1})) = a_{M+l}$. Furthermore, $a_{l} \le a_{M+l} \le a_{M}+a_{l}$. Then $\frac{a_{l}}{l} \le \frac{a_{l+M}}{l} \le \frac{a_{l}}{l}+ \frac{a_{M}}{l}$. By Lemma \ref{ent lim}, $\lim_{l \to \infty} \frac{a_{l}}{l}$ exists, and 
$$\lim_{l \to \infty} \frac{a_{l}}{l} \le  \lim_{l \to \infty}\frac{a_{l+M}}{l} \le \lim_{l \to \infty} (\frac{a_{l}}{l}+\frac{a_{M}}{l})= \lim_{l \to \infty} \frac{a_{l}}{l}+ 
\lim_{l \to {\infty}}\frac{a_{M}}{l} = \lim_{l \to {\infty}}\frac{a_{l}}{l}.
$$

It follows that  $$\lim_{l \to \infty}\frac{\log(N(\star_{i=1}^{l}H,\alpha^{lN+1}))}{l}= \lim_{l \to \infty}\frac{\log(N(\star_{i=1}^{M+l}H,\alpha^{(M+l)N+1}))}{l},$$ and thus, $\ent(G,\alpha)=h(\sigma^{N}, (\alpha^{MN+1}\times I^{\infty})^{*})$.

\end{proof}

\begin{theorem}
Suppose $H$ is a closed subset of $\Pi_{i=0}^{N}I_{i}$, $\mathbf{H} =\star_{i=1}^{\infty}H$, and $\sigma^{N}(\mathbf{H})=\mathbf{H}$.  If $\alpha=\{\alpha_{1}, \dots, \alpha_{n}\}$ is a minimal open cover of $I_{0}$ by open  intervals, then $\ent(H)=h(\sigma^{N})$.

\end{theorem}

\begin{proof}
 Since each open cover of $\mathbf{H}$ is refined by the grid cover $\alpha^{MN+1} \times I^{\infty}$ for some $M$ and minimal open cover $\alpha$ by intervals of $I_{0}$, the result follows.

\end{proof}

\begin{theorem}
Suppose $G$ is a closed subset of $I_{0} \times I_{1}$, $\mathbf{G} =\star_{i=1}^{\infty}G \ne \emptyset$, and $\sigma(\mathbf{G})=\mathbf{G}$. Then $\ent(\star_{i=1}^{k}G)=k \thinspace \ent(G)$.

\end{theorem}

\begin{proof} Supppose $k$ is a positive integer. Let $G$ be a closed subset of $I_{0} \times I_{1}$ such that $\mathbf{G} =\star_{i=1}^{\infty}G \ne \emptyset$, and $\sigma(\mathbf{G})=\mathbf{G}$, and let $H=\star_{i=1}^{k}G\subset \Pi_{i=0}^{k}I_{i}$.  Let $\alpha=\{\alpha_{1}, \dots, \alpha_{n}\}$ be a minimal open cover of $I_{0}$ by open  intervals. Then for each positive integer $m$, $\star_{i=1}^{m}H=\star_{i=1}^{m}(\star_{i=1}^{k}G)=\star_{i=1}^{mk}G$. Hence, $N(\star_{i=1}^{mk}G, \alpha^{mk+1})=N(\star_{i=1}^{m}H,\alpha^{mk+1})$.

\vphantom{}

Then 
\begin{align*}
\ent(G,\alpha)& =\lim_{m \to \infty}\frac{1}{m}\log N(\star_{i+1}^{m}G,\alpha^{m+1}) = 
\lim_{m \to \infty}\frac{1}{mk}\log N(\star_{i+1}^{mk}G,\alpha^{mk+1})\\
&  =\frac{1}{k}\lim_{m \to \infty}\frac{1}{m}\log N(\star_{i+1}^{mk}G,\alpha^{mk+1}) =\frac{1}{k}\ent(\star_{i=1}^{mk} G,\alpha^{mk+1}).
\end{align*}

Thus, for every minimal cover $\alpha$ by open intervals of $I_{0}$, $$k \thinspace \ent(G,\alpha)= \ent(\star_{i=1}^{mk}G,\alpha^{mk+1}).$$ The result follows.

\end{proof}

\section{Computation and Application of Topological Entropy}

In this section we compute the topological entropy for some closed subsets $G$ of $I^{2}$. While most authors prefer the Bowen approach (using $(n,\epsilon)$-spanning sets and $(n,\epsilon)$-separating sets), we find that "counting the boxes" (i.e., computing  $N(\star_{i=1}^m G,\alpha^{m+1})$) for $G$ a closed subset of $I^2$ and $\alpha$ a minimial open cover of $I$ by intervals) is often easy and natural. We also explore the relationship between $G$ and $\mathbf G=\star_{i=1}^\infty G,$ and investigate the interaction of the topology of $\mathbf G$ and the dynamics of the shift map $\sigma.$
\begin{exam} \label{exam square}
Suppose $G=I^2$. Then $\ent(G) = \infty$.

\end{exam}
\begin{proof} Suppose  $\alpha=\{\alpha_{1}, \dots, \alpha_{n}\}$ is a minimal open cover of $I_{0}$ by open  intervals. Then for each positive integer $m$, $N(\star_{i=1}^{m}G,\alpha^{m+1})=n^{m+1}$. Thus, 

$$\ent(G,\alpha^{2})=\lim_{m \to \infty}\frac{1}{m}\log N(\star_{i+1}^{m}G,\alpha^{m+1}) = 
\lim_{m \to \infty}\frac{1}{m}\log n^{m+1}$$ $$ =\lim_{m \to \infty}\frac{m+1}{m}\log n =\log n.$$

Then $$\sup_{\alpha}\ent(G,\alpha^2)=\sup_{\alpha} \log n = \infty.$$

\end{proof}

\noindent \textbf{Remark} While we computed the entropy of Example \ref{exam square} using our boxes, that $\ent(G)=\infty$ follows from \cite[Theorem 7.1]{KT}.

\begin{exam} \label{linetwopoints}
	Let $G$ denote the union of the diagonal from $(0,0)$ to $(1,1)$ and two points where $\left( x,y \right)$ is an arbitrary point in $I \times I$ such that $x \neq y$ and the second point is $\left( y,x \right)$. 
	Then, $\ent \left( G \right)=\log 2.$
\end{exam}
\begin{figure}
\centering
\begin{tikzpicture}[scale=0.7]
\begin{axis}[ 
unit vector ratio*=1 1 1, 
xlabel=$I_0$, xtick={0,0.5,1}, xticklabels={0,,1},
ylabel=\rotatebox{-90}{$I_1$}, 
ytick={0,0.5,1}, yticklabels={0,,1},
xmin=0,     
xmax=1.,   
ymin=0,     
ymax=1,
] 
\addplot [blue, ultra thick] coordinates {(0,0) (1,1)};
\addplot [blue, only marks, mark options={scale=1.2}] plot coordinates {(0.25,0.6) (0.6,0.25) (0.25,0.25) (0.6,0.6)};
\node at (axis cs:0.2,0.65)  {$(x,y)$};
\node at (axis cs:0.55,0.3)  {$(y,x)$};
\node at (axis cs:0.2,0.3)  {$(x,x)$};
\node at (axis cs:0.55,0.65)  {$(y,y)$};
\end{axis}
\end{tikzpicture}
\caption{Set $G$ from Example \ref{linetwopoints}}
\end{figure}
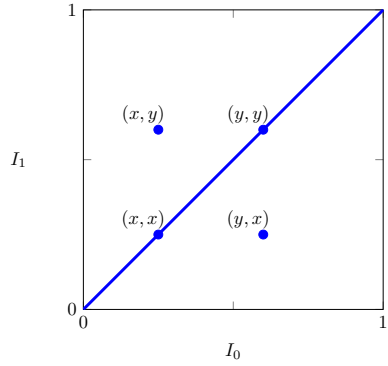

\begin{proof}
	The set $G$ is the union of the diagonal and two points, $G \star G$ is the union of the diagonal from $(0,0,0)$ to $(1,1,1)$ and six points, $\star_{i=1}^m G$ is the union of the diagonal and $2^{m+1}-2$ different points that don't lie on the diagonal.
	Namely, points in  $\star_{i=1}^m G$ that are not on the diagonal have $m+1$ coordinates and every coordinate is either $x$ or $y$, so there are $2^{m+1}-2$ different points which do not lie on diagonal.
	So, for intervals in $\alpha$ sufficiently small we have $$ N\left( G ,\alpha^2\right) = n+2,N \left( {G \star G},\alpha^3\right) = n+6, \ldots, N \left({{\star_{i=1}^m} G}, \alpha^{m+1}\right) = n+2^{m+1}-2.$$
	Hence, 
	$$
	\ent \left( G, \alpha \right)= \lim_{m \to \infty} \frac{1}{m} \log \left( n+2^{m+1}-2 \right)=\log 2.
	$$
	Therefore, $\ent \left( G \right) = \log 2.$
\end{proof}

In the example above, note that the topological entropy being $\log 2$ is completely determined by the four point subset $G'=\{ (x,y), (y,x), (x,x), (y,y) \}$, i.e., $\ent (G')=\log 2.$ In \cite{KT}, it was shown (Proposition 6.4) that it can never happen for a continuous function $f:X\to X$ on a compact metric space that the entropy $h(f)$ is determined by a finite set. They also showed (Proposition 6.1) that if $a\neq b $ in $ X $ compact metric, $F:X\to 2^X$ is upper semicontinuous, and $F(a)\supseteq \{a,b \}, F(b)\supseteq \{a,b \},$ then $h(F)\geq \log 2.$

\vphantom{}

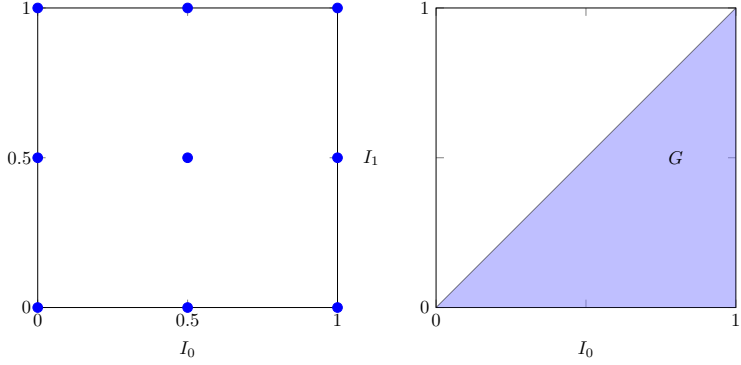
\begin{figure}[h!]
	\centering
	\begin{tikzpicture}[scale=0.7]
	\begin{axis}[ 
	unit vector ratio*=1 1 1,
	xlabel=$I_0$, xtick={0,0.5,1}, xticklabels={0,0.5,1},
	 ytick={0,0.5,1}, yticklabels={0,0.5,1},
	xmin=0,     
	xmax=1.,   
	ymin=0,     
	ymax=1,
	] 
	\addplot [blue, only marks, mark options={scale=1.3}] plot coordinates {(0,0) (0.5,0) (1,0) (0,0.5) (0.5,0.5) (1,0.5) (0,1) (0.5,1) (1,1)};
	\end{axis}
	\label{aa}
	\end{tikzpicture}
	\begin{tikzpicture}[scale=0.7]
	\begin{axis}[
	unit vector ratio*=1 1 1,
	xlabel=$I_0$, xtick={0,0.5,1}, xticklabels={0,,1},
	ylabel=\rotatebox{-90}{$I_1$}, ytick={0,0.5,1}, yticklabels={0,,1},
	xmin=0,     
	xmax=1,   
	ymin=0,     
	ymax=1,
	] 
	
	\addplot [smooth, opacity=0.5, domain=0:1,samples=20, fill=blue!50!white]
	{x}  \closedcycle ;

	\node at (axis cs:0.8,0.5)  {$G$};
	
	\end{axis}
	\end{tikzpicture}
	\caption{ The set on the left is $G_3$ from Example \ref{exam logn} and the set on the right is from Example \ref{exam triangle}}
\end{figure}

In the next example we find, for each positive integer $n>1,$ finite sets $G_n$ such that $\ent G_n = \log n.$ Of course, it is well known that for any positive number $r,$ there are continuous functions $f$ on the interval such that $h(f)=r.$
\begin{exam} \label{exam logn}
	For each $n \in \mathbb N$ there exists set $G_n \subseteq I \times I$ such that $\ent \left( G_n \right) = \log n.$
\end{exam}
\begin{proof}
	Let $n \in \mathbb N$ be arbitrary.
	We define $G_n$ in following way:
	$$
	G_n= \left \{ \left( \frac{k}{n-1}, \frac{l}{n-1} \right) : k,l \in \left \{ 0,1, \ldots, n-1 \right \} \right \}. 
	$$
	Set $G_n$ is union of $n^2$ points, $G_n \star G_n$ is union of $n ^3$ points, $\star_{i=1} ^m G_n $ is union of $n ^{m+1}$ different points so for intervals in $\alpha$ sufficiently small we have:
	$$ N \left(G, \alpha^2\right) = n ^2  ,N \left({G \star G}, \alpha^3\right) = n ^3 , \ldots, N \left( {{\star_{i=1}^m} G},\alpha^{m+1}\right) = n^{m+1}.$$
	We have
	$$ 
	\ent \left( G, \alpha \right) = \lim_{m \to \infty} \frac{1}{m} \log \left( n^{m+1} \right)=\log n.
	$$
	Therefore, $\ent \left( G \right)= \log n$
\end{proof}

The following example is discussed in \cite{KT}. They used Bowen's result (Theorem \ref{Bowen1}) to show that the entropy is $0.$ Since the proof of Bowen's theorem is nontrivial, quite delicate, and uses notation that is not defined and probably out of date,  we calculate the entropy directly. We explore this example in more depth later, as it turns out to be quite interesting, and ``just barely'' has entropy $0$.

\begin{exam} (The Triangle Example) \label{exam triangle}
	Let $G$ to be the set $\{ (x,y)\in I\times I: x \geq y   \}$. Then $\ent G=0.$
\end{exam}
\begin{proof}
	Let $\alpha = \{ \alpha_1,\ldots, \alpha_n  \}$ be an open cover of $I$ where $n\in \mathbb N$ is arbitrary and let $m\in \mathbb N , m \geq 2$ be arbitrary. We have $\star_{i=1}^m G=\{  (x_0,\ldots,x_m): x_i \geq x_{i+1}, 0\le i\le m-1 \}.$
	The elements of cover of $\star_{i=1}^m G$ are of the form $\alpha_{i_1} \times \alpha_{i_2} \times \ldots \times \alpha_{i_{m+1}},$ where $i_j \in \{ 1,2,\ldots,n \}.$ From the definition of the set $\star_{i=1}^m G$ it follows that $i_j \geq i_{j'},$ when $j' \geq j$.  Therefore, calculating the number $N(\star_{i=1}^{m}G,\alpha^{m+1})$ is equivalent to calculating the number of $(m+1)$-tuples $(i_1,i_2,\ldots,i_{m+1})$ such that $i_j \geq i_{j'},$ whenever $j' \geq j,$ where $i_j \in \{ 1,2,\ldots,n \}.$ 
	That number is equal to the binomial coefficient $\binom{m+1+n-1}{m+1}= \binom{m+n}{m+1}=\binom{m+n}{n-1}.$
	Therefore, $N(\star_{i=1}^{m}G,\alpha^{m+1})=\binom{m+n}{n-1}$.
	Now we have:
	$$
	\ent(G,\alpha)=\lim\limits_{m} \frac{\log \left( \binom{m+n}{n-1} \right)}{m} \leq \lim\limits_{m} \frac{\log \left( \frac{(m+n)^{n-1}}{ (n-1)!} \right)}{m}=\lim\limits_{m} \frac{(n-1) \log(m+n)}{m}=0.
	$$
	Therefore, $\ent(G,\alpha)=0$ and hence $\ent(G)=0.$
\end{proof}

The two following examples are immediate consequences of the previous example. The next three examples were considered by Ingram in \cite{I6}.
\\
\begin{exam} \cite[Example 2.14]{I7}
	Let $G= (\left \{ ( x,x)  \mid x \in I \right \}) \cup (\left \{ 1 \right \} \times I).$ Then $\ent \left( G \right) = 0.$
\end{exam}

\begin{exam} \cite[Example 2.2]{I7}
	Let $G= (\left \{ 0 \right \} \times [ 0,1 ]) \cup  ([0,1]  \times \left \{ 1 \right \}) . $ Then, $\ent \left( G \right)=0.$ 
\end{exam}

\begin{exam} \cite[Example 2.3]{I7} \label{exam ingram 2.3}
	Let $G = (\left \{ 0 \right \} \times [0,1]) \cup ([0,1] \times \left \{ 0 \right \}).$ Then $\ent \left( G \right) = \infty.$
\end{exam}

\begin{proof}
	Let  $L_1= \left \{ 0 \right \} \times [0,1]$ and  $L_2=[0,1] \times \left \{ 0 \right \}.$
	For arbitrary $m \in \mathbb N, \star_{i=1} ^m G$ contains $L_1 \star L_2 \star \ldots L_1 \star L_2 = [0,1] \times \left \{0 \right \} \times \ldots \times \left \{0 \right \} \times [0,1]$ if $m$ is even and $ L_1 \star L_2 \star \ldots L_2 \star L_1 = [0,1] \times \left \{0 \right \} \times \ldots \times [0,1] \times \left \{0 \right \} $ if $m$ is odd. Either way, $N \left( {{\star_{i=1}^m} G},\alpha^{m+1}\right) > n ^\frac{m}{2}.$
	Therefore, 
	$$
	\ent \left( G, \alpha \right)  \geq \lim_{m \to \infty} \frac{1}{m} \log \left(   n ^\frac{m}{2}  \right)=\frac{1}{2}\log n.
	$$
	and it follows $\ent \left( G \right) =\infty.$
\end{proof}

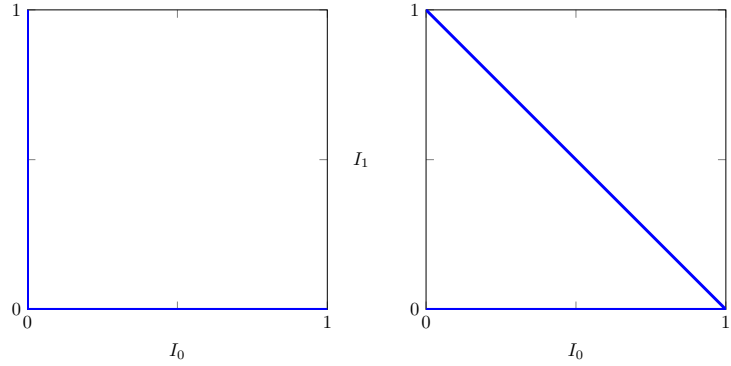
\begin{figure}
	\centering
\begin{tikzpicture}[scale=0.7]
\begin{axis}[ 
unit vector ratio*=1 1 1,
xlabel=$I_0$, xtick={0,0.5,1}, xticklabels={0,,1},
 ytick={0,0.5,1}, yticklabels={0,,1},
xmin=0,     
xmax=1.,   
ymin=0,     
ymax=1,
] 
\addplot [blue, ultra thick] coordinates {(0,1) (0,0) (1,0)};
\end{axis}
\end{tikzpicture}
\begin{tikzpicture}[scale=0.7]
\begin{axis}[ 
unit vector ratio*=1 1 1,
xlabel=$I_0$, xtick={0,0.5,1}, xticklabels={0,,1}, ylabel=\rotatebox{-90}{$I_1$}, 
ytick={0,0.5,1}, yticklabels={0,,1},
xmin=0,     
xmax=1.,   
ymin=0,     
ymax=1,
] 
\addplot [blue, ultra thick] coordinates {(0,0) (1,0) (0,1)};
\end{axis}
\end{tikzpicture}
\caption{The sets $G$ from Examples \ref{exam ingram 2.3} and \ref{exam maribor monster}}
\end{figure}

The Maribor Monster Example below has been studied by several researchers, as it has quite interesting topology. (See \cite{BCMM1}, \cite{I6}, \cite{KN}, for example.) The continuum that forms the inverse limit is a $\lambda$-dendroid that is both $\frac{1}{2}$-indecomposable and hereditarily decomposable.  The dynamical behavior of $\sigma$ on the inverse limit is a bit easier to understand. There is one invariant arc $K=\{(x,1-x,x,1-x,...): x \in I \}$ that repels all other points, and there is an attracting invariant Cantor set $C$ (a subshift of the full 2-shift) that actually eventually ``absorbs'' all points not on $K$. (See \cite{KN} for definitions and details.) It is the invariant Cantor set, which is generated by the finite set $C'=\{(0,0),(1,0),(0,1)\}$ in $I^2$, that determines the entropy here. 

\begin{exam} \label{exam maribor monster} (Maribor Monster Example) 
	Let $G=( I \times \left \{ 0 \right \}) \cup (\left \{ (x,1-x) \mid x \in I \right \}).$ Then $\ent (G) =\frac{1+\sqrt{5}}{2}$, i.e., the so-called "golden ratio".
\end{exam} 

\begin{proof}
	Let us denote with $L_1$ line from $(1,0)$ to $(0,0)$ and with $L_2$ line from $(1,0)$ to $(0,1).$ We have $G= L_1 \cup L_2.$
	In arbitrary product $\star_{j=1} ^m L_{i_j}, i_j \in\left \{ 1,2 \right \}, $ let $i_0$ be first index such that $L
_{i_0}=L_1.$ Next coordinate has to be $0$ and after that we have only zeros and ones such that we cannot have two neighboring ones.
	If $L_i=L_2, \forall i$ then the product is arc from $(0,1,0,\ldots)$ to $(1,0,1,\ldots).$
	So we have 
	\begin{align*}
	N \left( {{\star_{i=1}^{m-1}} G},\alpha^{m}\right) + N \left( {\star_{i=1}^{m-2} G},\alpha^{m-1}\right) \leq & N\left( {{\star_{i=1}^m} G} ,\alpha^{m+1}\right) \\
	\leq & N \left({{\star_{i=1}^{m-1}} G}, \alpha^{m}\right) + N \left( {{\star_{i=1}^{m-2}} G},\alpha^{m-1}\right) + n
	\end{align*}
	and hence $ nF_{m+2} \leq  N \left( {{\star_{i=1}^m} G},\alpha^{m+1}\right) \leq n(F_{m+3}-1),$ where $F_m$ is $m-th$ Fibonacci number. 
	Therefore, we have that
	$$
	\ent \left( G, \alpha \right) = \lim_{m \to \infty} \frac{1}{m} \log N\left( {{\star_{i=1}^m} G} ,\alpha^{m+1}\right)= \lim_{m \to \infty} \frac{1}{m} \log n(F_{m+3} - 1)=\frac{1+\sqrt{5}}{2}.
	$$
	and $\ent (G) = \log \frac{1+\sqrt{5}}{2}. $
\end{proof}
\begin{figure}[h!]
	\centering
	\begin{tikzpicture}[scale=0.7]
	\begin{axis}[ 
	unit vector ratio*=1 1 1,
	xlabel=$I_0$, xtick={0,0.25,0.5,0.75,1}, xticklabels={0,0.25,$\cdots$,0.75,1},
	ylabel=\rotatebox{-90}{$I_1$}, ytick={0,0.25,0.5,0.75,1}, yticklabels={0,0.25,$\cdots$,0.75,1},
	xmin=0,     
	xmax=1.,   
	ymin=0,     
	ymax=1,
	] 
	\addplot [blue, only marks, mark options={scale=1.3}] plot coordinates {(0,0) (0.25,0) (0.75,0) (1,0) (0,0.25) (0,0.75) (0,1)};
	\node at (axis cs:0.5,0.12)  {$a+1$ points};
	\node at (axis cs:0.5,0.05) {$\overbrace{\qquad \qquad \qquad \qquad \qquad \qquad \qquad \qquad}$};
	\end{axis}
	\end{tikzpicture}
	\caption{Set $G$ from Example  \ref{exam points} }
\end{figure}
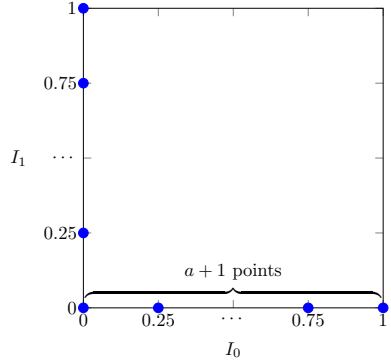
\begin{exam} \label{exam points}
	Let $a \in \mathbb N, a >1$ be arbitrary and let $G_a=\left \{ (\frac{k}{a-1},0) \mid k \in \left\{ 0,\ldots, a-1\right\}  \right \} \cup \left \{ (0, \frac{k}{a-1}) \mid k \in \left\{ 1,\ldots, a-1\right\}  \right \} \subseteq I^2 .$ Then $\ent(G_a)=\log \frac{1+\sqrt{1+4a}}{2}.$ 
\end{exam}
\begin{proof}
	Let us denote number of points in $\star_{i=1}^m G_a$ with $N_m.$ 
	\\We prove
	$N_m=N_{m-1}+a \ N_{m-2}.$ 
	\\Proof is combinatorial: let us observe arbitrary $(m+1)-\textrm{tuple}$ in $\star_{i=1}^m G_a$. If we have $0$ on the first coordinate, second can be any number from $\left \{ 0, \frac{1}{a}, \frac{2}{a},\ldots, 1 \right  \}$ so we can get any $m-\textrm{tuple}$. If we have non-zero as first coordinate, second coordinate has to be zero, third can be anything (as above) so we can get any $(m-1)-\textrm{tuple}$. 
	Therefore, we get reccurence relation $N_m=N_{m-1}+a \ N_{m-2}$ with initial values $N_1=2a+1$ and $N_2=(a+1)^2+a.$ Solving it using characteristical polynomial $x^2-x-a=0$ we get
	$$
	N_m= a_0 \left(\frac{1+\sqrt{1+4a}}{2} \right)^m + b_0 \left(\frac{1-\sqrt{1+4a}}{2} \right)^m
	$$
	where $a_0$ and $b_0$ are positive real numbers obtained from initial values.\\
	Therefore, for intervals in $\alpha$ sufficiently small we have
	$$
	N\left( {{\star_{i=1}^m} G_a} ,\alpha^{m+1}\right) = N_m
	$$
	and 
	$$
	\lim_{m \to \infty}  N \left( {{\star_{i=1}^m} G_a},\alpha^{m+1}\right) = \lim_{m \to \infty}  N_m .
	$$
	Now, 
	$$
	\ent (G_a, \alpha)=\lim_{m \to \infty} \frac{\log N_m}{m}.
	$$
	By simple calculations we get 
	$$
	\ent (G_a, \alpha)=\frac{1+\sqrt{1+4a}}{2}.
	$$
\end{proof}

\begin{exam} \label{exam11}
	Let $G=\left\{ (x_0,x_1) \in I_0 \times I_1: x_1 \leq x_0^2 \right\}$ and let $bL=([0,1] \times \left\{0 \right\}) \cup (\left \{ 1\right\} \times [0,1]).$ Then $\ent(G)=\ent(bL)=0.$	
\end{exam}

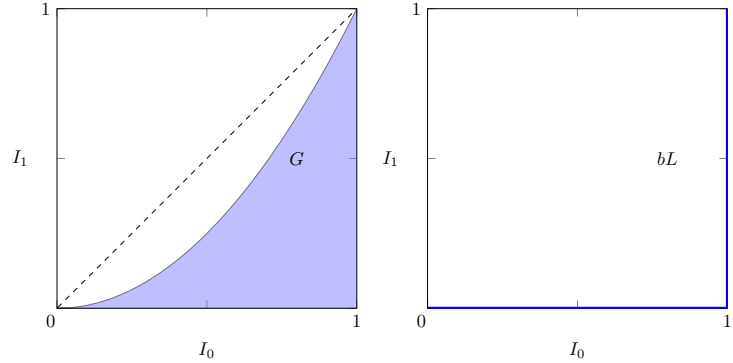
\begin{figure}[h]
	\begin{tikzpicture} [scale=0.7]
	\begin{axis}[
	unit vector ratio*=1 1 1, 
	xlabel=$I_0$, xtick={0,0.5,1}, xticklabels={$0 \ \ $,,1},
	ylabel near ticks,  ylabel=\rotatebox{-90}{$I_1$}, ytick={0,0.5,1}, yticklabels={
		,,1},
	xmin=0,     
	xmax=1,   
	ymin=0,     
	ymax=1,
	] 
	
	\addplot [smooth, opacity=0.5, domain=0:1,samples=20, fill=blue!50!white]
	{x^2}  \closedcycle ; 
	
	\addplot [dashed, black]{x};
	\node at (axis cs:0.8,0.5)  {$G$};
	
	\end{axis}
	\end{tikzpicture}
	\begin{tikzpicture}[scale=0.7]
	\begin{axis}[ 
	unit vector ratio*=1 1 1, 
	xlabel=$I_0$, xtick={0,0.5,1}, xticklabels={$0 \ \ $,,1},
	ylabel near ticks, ylabel=\rotatebox{-90}{$I_1$}, ytick={0,0.5,1}, yticklabels={,,1},
	xmin=0,     
	xmax=1.,   
	ymin=0,     
	ymax=1,
	] 
	\addplot [blue, ultra thick] coordinates {(0,0) (1,0)  (1,1)};
	\node at (axis cs:0.8,0.5)  {$bL$};
	\end{axis}
	\end{tikzpicture}
	\caption{Sets $G$ and $bL$ from the Example \ref{exam11}}
\end{figure}

\begin{proof}
	The sets $G$ and $bL$ are subsets of the set $G$ in Triangle Example so it follows from that and Proposition \ref{ent subset} (c).
\end{proof}

Before stating the proposition we give a new notion.
If $H$ is closed subset of $[0,1]^{n+1},$ define $\pi_{\{0,n\}}$ to be the map from $H$ to $[0,1]\times [0,1]$ defined by $\pi_{\{0,n\}}(x)=(x_0,x_n)$ for $x=(x_0,x_1,\ldots,x_n)\in H.$ For $G$ a closed subset of $[0,1]\times [0,1]$, and $n$ positive integer, let $G^{0,n}=\pi_{\{0,n\}}(\star_{i=1}^n G).$ Therefore, $G^{0,n}$ is a closed subset of $[0,1]\times [0,1].$

\vphantom{}

With $\lim_{H_d}$ we denote limit with respect to the Hausdorff metric. Let us recall that Hausdorff metric:
$$H_{d}\left( A,B\right) =\max \left\{ \sup_{a\in A}d\left( a,B\right) ,\sup_{b\in B}d\left( b,A\right) \right\} .$$

\begin{prop}
	Let $G$ be a connected and closed subset of $[0,1]^2$ such that $\lim_{H_d} G^{0,n}=bL$ where $bL$ is from the previous example. Then $\ent(G)=0.$	
\end{prop}
\begin{proof}
	We have $\lim_{H_d} G^{0,n} = bL$  i.e. for each $\epsilon >0$ there is positive integer $n_0$ such that for every positive integer $n, n \geq n_0$ it follows $H_d (G^{0,n},bL) < \epsilon$.
	We divide the proof in several steps:
	\begin{itemize}
		\item[(i)] $G$ does not contain any point on diagonal not equal to $(0,0)$ and $(1,1)$. 
		\\ Assume the contrary, i.e. if it does contain point $(x,x), x \neq 0, 1$, then $G^{0,n}$ also contains that point, for all $n \in \mathbb N$. 
		Then, $d((x,x),bL)= \left \{ 
		\begin{array}{cl}  
		x, & x \leq \frac{1}{2} \\ 
		1-x, & x >\frac{1}{2} 
		\end{array} 
		\right.$  but either way it is greater than 0.
		
		Therefore, we get  $H_d(G^{0,n},bL) \geq d((x,x),bL)>0, \forall n \in \mathbb N,$ hence we cannot have $\lim_{H_d} G^{0,n} = bL$.
		
		\item[(ii)] $\pi_0 (G)=[0,1]$ and $\pi_1 (G)=[0,1].$
		\\ Suppose $\pi_0 (G) = J_0$ where $J_0$ is a closed and proper subset of $[0,1].$ Then, $G^{0,n} \subseteq  J_0 \times [0,1] \subset [0,1] \times [0,1]$ for all $n \in \mathbb N.$ Since $bL$ contains $[0,1] \times \left\{ 0 \right \}$, there exists a point $(x,0) \in bL \setminus \left( J_0 \times \left\{ 0 \right \} \right)$ such that $d((x,0), J_0)=d_0 >0$ (because $J_0$ is closed). Now we have that $0< d_0=d((x,0), J_0 \times [0,1] ) \leq d((x,0), G^{0,n})$ and therefore $H_d (G^{0,n},bL) \geq d_0 >0.$ So we get $\lim_{H_d} G^{0,n} \neq  bL$ which is contradiction. Similarly, $\pi_1 (G)=[0,1].$
		
		\item[(iii)] $G$ doesn't contain point $(x_0,y_0)$ above the diagonal.
		\\
		Suppose that $G$ contains a point $(x_0,y_0)$ above the diagonal, i.e. $y_0>x_0.$
		The set $G$ is closed, connected and by (ii),  $\pi_0 (G)=[0,1]$ and $\pi_1 (G)=[0,1],$ therefore $G$ intersects the diagonal in some point $(x,x), x \in [0,1].$ By (i) it follows that $x$ is equal to $0$ or $1$. 
		\\
		Suppose that $x=0.$ Since $G$ is connected and closed, it contains a subcontinuum connecting points $(x_0,y_0)$ and $(0,0)$ which we denote with $K$. Since $K$ is a subcontinuum, for each positive integer $n$ there exists finite sequence of points $x_{n-1} < x_{n-2} <\ldots < x_0 < y_0$ in $[0,1]$ such that $(x_{n-1}, x_{n-2}, \ldots , x_0, y_0) \in \star_{i=1}^{n} K.$ Therefore, since $K \subset G,$ we have $H_d(G^{0,n},bL)  \geq H_d \left( K^{0,n} , bL\right) \geq  d((x_0,y_0),bL) > 0, $ hence we cannot have $\lim_{H_d} G^{0,n} = bL.$
		\\
		If we suppose that $x=1,$ we get the contradiction in the same way.

	\end{itemize}
	
	From (i) and (iii) it follows that set $G$ is under the diagonal, except points $(0,0)$ and $(1,1)$ i.e. $$(\forall (x,y) \in G \setminus \left \{ (0,0),(1,1)\right \}, \ y<x.) \ (*)$$
	
	Now we have that $G$ is a subset of the set in the Triangle Example and therefore $\ent(G)=0.$
	
\end{proof}

\vphantom{}

Now we return to the Triangle Example and its properties. Let  $M$ denote the Mahavier product $\star_{i=1}^{\infty}G$ produced by this example. It was noted in \cite{KT} that $M$ contains copies of the Hilbert cube. It may be homeomorphic to it. 

\vphantom{}

The dynamics of $\sigma$ on $M$ are quite easy to describe. If $\mathbf{x}= (x_0,x_1,\ldots)$ is a point in $M$, then 
$x_0 \ge x_1 \ge \cdots$ and, as a consequence, if $x^* = \inf \{x_i: i \ge 0\}$, then $\mathbf{x}, \sigma(\mathbf{x}), \sigma^2(\mathbf{x}), \ldots$ converges to the point $\mathbf{x}^*=(x^*,x^*,\ldots)$. And if you look backwards to where points ``come from'', so, consider the double-sided inverse limit $\star_{i=-\infty}^{\infty}G$, then if $x^* = \sup \{x_{-i}: i \ge 0 \}$, then $\mathbf{x}, \sigma^{-1}(\mathbf{x}), \sigma^{-2}(\mathbf{x}), \ldots$ converges to the point $\mathbf{x}^*=(x^*,x^*,\ldots)$. 

\vphantom{}

In the Triangle Example  points can ``slow down'' and hang around as long as they wish before continuing to their destination. The propositions below make this precise. There are many different definitions of chaotic maps (having positive entropy is one of them), but sensitive dependence on intial conditions is a property that virtually all chaotic maps, or maps chaotic on some subset, share. The usual definition of sensitive dependence on inititial conditions is given below. The Triangle Example doesn't have that, but it has a sort of weak version of it, and also a sort of weak version of the specification property. We define weak sensitivity and show that the Traingle Example has weak sensitivity on a subset.

\vphantom{} 

 Suppose $X$ is a compact metric space and $f:X \to X$ is a map. Then $f$ has \textit{sensitive dependence on initial conditions (SDIC)}  if there is some $\delta>0$ such that for each $\epsilon > 0$ and each point $x$ in $X$, there are a point $y$ in $X$ with $d(x,y)< \epsilon$, and an integer $n>0$ such that $d(f^n(x),f^n(y))>\delta$.

\vphantom{}

 Suppose $X$ is a compact metric space and $f:X \to X$ is a map. Then $f$ has \textit{ weak sensitivity (WS) } at the point $x$ if there is some $\delta>0$ such that for each $\epsilon > 0$ there are a point $y$ in $X$ with $d(x,y)< \epsilon$ and an integer $n>0$ such that $d(f^n(x),f^n(y))>\delta$. $f$  has \textit{weak sensitivity (WS) on the set}  $A \subset X$ if  there is some $\delta>0$ such that for each $\epsilon > 0$ and each $x \in A$ there are a point $y$ in $X$ with $d(x,y)<\epsilon$ and an integer $n>0$ such that $d(f^n(x),f^n(y))>\delta$.

\begin{prop} 
Suppose $1>\delta >0$, and $M_{\delta}$ is the closed subset $\{ \mathbf{y}=(y_0,y_1, \ldots) \in M:  y_i \ge \delta \text{ for each } i \}$. Then $\sigma$ has WS at each point of $M_{\delta}$.

\end{prop}

\begin{proof}
	Let $\epsilon>0$ and $ x= (x_0,x_1,\ldots) \in M_\delta$ be arbitrary. Suppose $y =(x_0,x_1,\ldots, x_{i_0},0,\ldots)$ where $i_0$ is positive integer such that $i_0 > \log_2 \frac{1}{\epsilon}.$ We have $d( x,  y)=\sum_{i=0}^\infty \frac{|x_i-y_i|}{2^i}=\sum_{i=i_0+1}^\infty \frac{|x_i|}{2^i}\leq \sum_{i=i_0+1}^\infty \frac{1}{2^i}=\frac{1}{2^{i_0}}<\epsilon$ and $\sigma^{i_0+1}(y) =\textbf 0$ (where $\textbf{0}=(0,0,0,...)$). Hence, $d(\sigma^{i_0+1}( x), \sigma^{i_0+1}( y))=d((x_{i_0+1},x_{i_0+2},\ldots),\mathbf 0)=\sum_{i=0}^\infty \frac{|x_{i_0+1+i}|}{2^i}\geq \delta \sum_{i=0}^\infty \frac{1}{2^i}=2\delta>\delta$ and we are done.
	
\end{proof}

 \begin{prop} Let $M'=\cup_{i=1}^{\infty} M_{1/i}$. Then $M'$ is dense and connected in $M$ and is a union of closed sets each of which has WS. However, each $M_{1/i}$ has empty interior and the $M_{1/i}$'s are nested, so $M'$ is a first category $F_{\sigma}$-set. (Also, note that $M'$ itself does not have WS.)

\end{prop}

\vphantom{}

\textbf{Notation.}  For $0 \le x \le 1$, let $\overline{x}=(x,x,x,\ldots)$, and note that $\overline{x} \in M$. Let the \textit{diagonal}, denoted $\Delta$, be the following subset of $M$:

$$ \Delta = \{ \overline{x}: 0 \le x \le 1 \}.$$

\vphantom{}
 
\textbf{Notation.} Suppose $z \in [0,1]$ . Let $z^m$ denote the point  $\overbrace{z,z,z,\ldots,z}^m$ in $I^m$. If $z_i \in [0,1]$ for each positive integer $i$, and  $m_i$ is a positive integer for each positive integer $i$, let $z_1^{m_1} \oplus z_2^{m_2} \oplus \cdots$ denote the concatenation 

$$(\overbrace{z_1,z_1,\ldots,z_1}^{m_1},\overbrace{z_2,z_2,\ldots,z_2}^{m_2}, \cdots).$$

\vphantom{}

The following shows that we have a weak form of a  shadowing or specification property with the triangle example.

\begin{prop} 
	Suppose $1 > z_1 > z_2> \cdots \ge 0$, and $\{m_i \}_{i=1}^\infty$ is a sequence of natural numbers greater than $1$. Then if $x=z_1^{m_1} \oplus z_2^{m_2} \oplus \cdots$, $d(x,\overline{z_1})\le 2^{-m_1+1}$, $d(\sigma^{m_1}(x),\overline{z_2})\le 2^{-m_2+1}$, $d(\sigma^{m_1+m_2}(x),\overline{z_3})\le 2^{-m_3+1}$, and so on. Suppose, in addition, $\{n_j \}_{j=1}^\infty$ a sequence of natural numbers greater than $1$, and $y=z_1^{m_1+n_1} \oplus z_2^{m_2+n_2} \oplus \cdots$. Then $d(\sigma^j(y),\overline{z_1})\le 2^{-m_1+1}$ for $0 \le j \le n_1$,  $d(\sigma^{m_1+n_1+j}(y),\overline{z_2})\le 2^{-m_2+1}$ for $0 \le j \le n_2$, and so on.
	
\end{prop}

\begin{proof} Let $\pi_i(x)=x_i, i \geq 0.$ The proof is straightforward: Note that $d(x,\overline{z_1}) = \Sigma_{i=0}^{\infty}\frac{|x_i -z_1|}{2^i}=\Sigma_{i=m_1}^{\infty}\frac{|x_i -z_1|}{2^i}\le 2^{-m_1+1}$, and then that $d(\sigma^{m_1}(x),\overline{z_2})= \Sigma_{i=0}^{\infty}\frac{|x_{m_1+i} -z_2|}{2^i}=\Sigma_{i=m_2}^{\infty}\frac{|x_{m_1+i} -z_2|}{2^i} \le 2^{-m_2+1}$. This pattern continues. The second part is just an extension of the first.
	
\end{proof}

	


\begin{prop} In the Triangle Example $M$, for each $0<a<1$, the open set $V_a=M \cap ((a,1] \times [0,a) \times I^{\infty})$ is a simple wandering set for $\sigma$. For each $n \in \mathbb{N}$, the set $M \cap (I^n \times V_a)$ is also a simple wandering set. The nonwandering set $\Omega_{\sigma}= \Delta= \{ \overline{x}=(x,x,x,\ldots): 0\le x \le 1 \}$, and so the wandering set $M \setminus \Omega_{\sigma}= M \setminus\Delta =\{x=(x_0,x_1,\ldots) \in M: \text{ for some } i, x_i \ne x_{i+1} \}.$

\end{prop}

\begin{proof}

Suppose $V_a \cap \sigma^m(V_a) \ne \emptyset$ for some $m \ne 0$. If $m >0$, then there is some $x=(x_0,x_1,x_2,\ldots) \in V_a$ such that $\sigma^m(x) \in V_a$. But then $x_0 \in (a,1]$ and $x_1 \in [0,a)$ and $x_0 \ge x_1 \ge x_2 \ge \cdots$. Since $\sigma^m(x)=(x_m,x_{m+1},\ldots)$, $x_m <a$, so $x_m \notin (a,1]$. This is a contradiction. 

\vphantom{}

Suppose $m <0$. Then there is some $x=(x_0,x_1,x_2,\ldots) \in V_a \cap \sigma^m(V_a)$. Then $\sigma^{-m}(x)=(x_{-m},x_{-m+1},\ldots) \in V_a$. But $x_0 \in (a,1]$, $x_1 \in [0,a)$, and $x_0 \ge x_1 \ge x_2 \ge \cdots$, and again we have a contradiction. Thus, $V_a$ is a wandering set for each $0<a<1$.

\vphantom{}

It is straightforward to show that  $M \cap (I^n \times V_a)$ for each $n>0$ is also wandering. If $x=(x_0,x_1,\ldots) \notin \Delta$, then there is least $i \ge 0$ such that $x_i >x_{i+1}$. There is some $0<a<1$ such that $x_i >a >x_{i+1}$. If $i=0$, then $x \in V_a$. If $i>0$, then $x \in I^{i}\times V_a$. Hence $M\setminus\Delta$ is a subset of the wandering set. 

\vphantom{}

 Clearlly, $\Delta \subset \Omega_{\sigma}$.Then  $\Delta = \Omega_{\sigma}$ and  $M \setminus\Delta = M \setminus \Omega_{\sigma}$

\end{proof}

 What happens if we add a point to the set $G$ in $[0,1]^2$ above? The sequence of lemmas below discusses this.
 
\vphantom{}

\begin{lemma}

Let $H' = \{(x,0): 0 \le x \le 1 \} \cup \{(1,y): 0 \le y \le 1 \}$ and $H=H' \cup \{(0,1)\}$. Then $H' \subset G$ and  $H$ has $\infty$ entropy.

\end{lemma}

\begin{proof} That $H' \subset G$ is obvious. So let us prove that the entropy of $H$ is $\infty$. Suppose that $n$ is a positive integer. Let $L_n= \{(i/n,0):0\le i \le n \} \cup \{(1,i/n): 0 \le i \le n\} \cup \{(0,1)\}$. Then $L_n \subset H$ and $|L_n|=2n+2$, where  $|L_n|$ denotes the cardinality of this finite set.

\vphantom{}

Then $\mathbf{L}_n=\star_{i=1}^{\infty} L_n \subset \star_{i=1}^{\infty} H$ is a Cantor set. In fact, $\mathbf{L}_n=\{(s_0,s_1,\ldots): s_i\in \{j/n:0 \le j \le n\}; \text{ if } s_i=k/n, \text{ where } 0 < k < n, \text { then } s_{i+1} =0; \text{ if } s_i=1, \text{ then } s_{i+1} = k/n \text{ where } 0 \le k \le n , \text{ and if } s_i=0, \text{ then } s_{i+1}=0 \text{ or } s_{i+1}=1 \}$. Note that the point  $(0,1,i/n,0) \in \star_{j=1}^{3} L_n $ for $0 \le i \le n$. Let $A_n= \{ (0,1,k/n,0): 0 \le k \le n \} \subset \star_{j=1}^3 L_n$. Note that $N(A_n, \alpha^{3+1})=n+1$ (as long has $\alpha$ has sufficiently small intervals and is chosen so that no $\frac{k}{n}$ is in 2 intervals), and, in general, for $m>0$, $N(\star_{i=1}^m A_n, \alpha^{3m+1})=(n+1)^m \le N(\star_{j=1}^{3m} L_n, \alpha^{3m+1})$.

\vphantom{}

Then 

\begin{align*}
	\ent(\mathbf{L}_n,\alpha)&=\lim_{m \to \infty} \frac{\log(N(\star_{j=1}^m L_n,\alpha^{m+1})}{m} =\lim_{3m \to \infty} \frac{\log(N(\star_{j=1}^{3m} L_n,\alpha^{3m+1})}{3m} \\ 
	& \ge \lim_{3m \to \infty} \frac{\log(N(\star_{i=1}^m A_n,\alpha^{3m+1})}{3m} = \lim_{3m \to \infty} \frac{\log( n+1)^m}{3m} \\
	& =\lim_{3m \to \infty} \frac{m\log( n+1)}{3m}=\frac{\log(n+1)}{3}.
\end{align*}

Then for each $\alpha$ with sufficiently small and carefully chosen intervals, $\ent(L_n,\alpha)= \frac{\log(n+1)}{3}$, so $\ent(L_n)= \frac{\log(n+1)}{3}$. Since $\ent(H) \ge \ent(L_n) \ge \frac{\log(n+1)}{3}$ for each $n$, the result follows.

\end{proof}

\vphantom{}

\begin{lemma}

Suppose $0\le p<q \le1$. Let $H'^{(p,q)} = \{(x,p): p \le x \le q \} \cup \{(q,y): p \le y \le q \}$ and $H^{(p,q)}=H'^{(p,q)}\cup \{(p,q)\}$. Then $H'^{(p,q)} \subset G$ and  $H^{(p,q)}$ has $\infty$ entropy.

\end{lemma}

\begin{proof} This is similar to the last result. We let, for each $n$, $L_n^{(p,q)}=\{(p+i \Delta t,p): 0 \le i \le n \} \cup \{(q,p+i\Delta t): 0 \le i \le n\} \cup \{(p,q)\},$ where $\Delta t=\frac{q- p}{n}$. Then it follows, as above, that $\ent(L_n^{(p,q)})=\frac{\log(n+1)}{3}$. Since $L_n^{(p,q)} \subset H^{(p,q)}$, the result follows.

\end{proof}

Then we have the following result:

\begin{theorem} 

Suppose $0\le p<q \le1$. Then $\ent(G \cup \{(p,q)\}) = \infty$, while  $\ent(G ) = 0$.

\end{theorem}
\vphantom{}

 Suppose $X$ is a metric space, $f:X \to X$ is a map. 

\begin{enumerate}

\item A point $x$ in $X$ has \textit{ period} $n$ if $f^n(x)=x$. A point $x$ in $X$ has \textit{prime period} $n$ if $f^n(x)=x$, but $f^j(x) \ne x$ for $0<j<n$. The point $x$ is \textit{periodic}. 

\item $f$ is \textit{transitive} if for each pair $U,V$ of open sets in $X$, there is $n$ such that $f^n(U) \cap V \ne \emptyset$.

\end{enumerate}

\begin{theorem} 
Let $G \cup \{(0,1)\}=G^+$. Then $\mathbf{G}^+=\star_{i=1}^{\infty}G^+$ under the action of $\sigma$ has (1) a dense set of periodic points, (2) has periodic points of all periods (nontrivially), and (3) $\sigma$ is transitive. 

\end{theorem}

\begin{proof} 

\begin{enumerate}

\item Suppose $x=(x_0,x_1,\ldots) \in \mathbf{G}^+$, and $x \in U=U_0 \times U_1 \times \cdots \times U_n \times I^{\infty}$, a basic open set. Let $z=(x_0,x_1, \ldots,x_n)$. Then $w=z \oplus (0,1) \oplus z \oplus (0,1) \oplus \cdots \in U \cap \mathbf{G}^+$, and $w$ is periodic. Hence the set of periodic points of $\mathbf{G}^+$ is dense in $\mathbf{G}^+$.

\item $\Delta \subset \mathbf{G}^+$, so $\mathbf{G}^+$ has fixed points (period one points). The point $(0,1,0,1,0,1,\ldots) \in\mathbf{G}^+$ has period 2. If $0<x<1$, then $(x,0,1,x,0,1,\ldots) \in \mathbf{G}^+$ has prime period 3. If $0<x<y<1$, then the point $(y,x,0,1,y,x,0,1,\ldots.) \in \mathbf{G}^+$ has prime period 4. And, so on. Hence $\mathbf{G}^+$ has prime period $n$ points for each positive integer $n$.

\item Suppose $U,V$ are open sets in $\mathbf{G}^+$. Without loss of generality, we can assume that $U,V$ are basic open sets. So, let $U= \mathbf{G}^+ \cap (U_0\times U_1 \times \cdots \times U_n \times I^{\infty})$ and let $V= \mathbf{G}^+ \cap (V_0\times V_1 \times \cdots \times V_m \times I^{\infty})$. Suppose $x=(x_0,x_1,\ldots) \in U$ and $y=(y_0,y_1,\ldots) \in V$. Then consider the point $z=(x_0,x_1,\ldots x_n) \oplus (0,1) \oplus y$. The point $z \in \mathbf{G}^+$, $z \in U$, and $\sigma^{n+3}(z)=y$. Thus, $\sigma^{n+3}(U) \cap V \ne \emptyset$. Hence $\sigma$ is transitive.

\end{enumerate}

\end{proof}

\begin{corollary}

The map $\sigma$ on $\mathbf{G}^+$ has sensitive dependence on initial conditions and  is chaotic in the sense of Devaney. 

\end{corollary}

\begin{proof} That $\sigma$ has sensitive dependence on initial conditions follows from \cite{BBCDS}. Devaney's definition of chaos requires (1) transitivity, (2) dense set of periodic points, and (3) sensitive dependence on initial conditions. (The paper cited above shows that SDIC follows from transitivity and dense set of periodic points.)

\end{proof}

\textbf{Notation.} Suppose $0<p<q<1$. Let $T^{(p,q)}=\{ (x,y): q \ge x \ge y \ge p \} \cup \{(p,q)\}$. Let $\mathbf{T}^{(p,q)}= \star_{i=1}^{\infty}T^{(p,q)}$.

\begin{lemma}  Then $\sigma(\mathbf{T}^{(p,q)})= \mathbf{T}^{(p,q)}$ and $\sigma$ on  $\mathbf{T}^{(p,q)}$   (1) admits a  dense set of periodic points, (2) admits periodic points of all periods (nontrivially), and (3) is transitive. Hence, $\sigma$ on  $\mathbf{T}^{(p,q)}$ is chaotic in the sense of Devaney. Furthermore, $\sigma|\mathbf{T}^{(p,q)}$ has entopy $\infty$.

\end{lemma}

\begin{proof} Using the same (slightly adjusted) arguments as those for Theorem 1, Theorem 2, and Corollary 1 shows this.

\end{proof}

Let $\mathbf{W}^{(p,q)}:=\star_{i=1}^{\infty} (G \cup \{(p,q)\})$, which contains $\mathbf{T}^{(p,q)}$. Under $\sigma$, $\mathbf{T}^{(p,q)}$ is invariant and closed in $\mathbf{W}^{(p,q)}$. But do all points of $\mathbf{W}^{(p,q)} \setminus (\mathbf{T}^{(p,q)} \cup \Delta)$ wander?

\vphantom{}

\begin{enumerate}

\item Suppose $x$ is in $\mathbf{W}^{(p,q)},$ $x$ is eventually in $\mathbf{T}^{(p,q)}$, but is not in $\mathbf{T}^{(p,q)}$. Then there is least $N \ge 0$ and $z=(z_0,z_1,\ldots) \in \mathbf{T}^{(p,q)}$ such that $x=(x_0,\ldots,x_N) \oplus z$, and $x_i>q$, $0 \le i \le N$. Suppose $x_N>a>q$. Then $\mathbf{x} \in V_a=\mathbf{W}^{(p,q)} \cap (I^{N} \times (a,1] \times [0,a) \times I^{\infty})$. Hence, $x_N \in (a,1]$ and $z_0=x_{N+1} \in [0,a)$. 

\vphantom{}

If $y=(y_0,y_1,\ldots) \in \sigma^{-1}(V_a)=\mathbf{W}^{(p,q)} \cap (I^{N+1} \times (a,1] \times [0,a) \times I^{\infty})$, then $y \notin V_a$, because $y_i > q$ for $0 \le i \le N+1$, and $y_{N+1} \notin [0,a)$. Likewise, $\sigma^{-m}(V_a) \cap V_a = \emptyset$ for $m>0$. 

\vphantom{}

If $y=(y_0,y_1,\ldots) \in \sigma(V_a)=\mathbf{W}^{(p,q)} \cap (I^{N-1} \times (a,1] \times [0,a) \times I^{\infty})$ (provided $N>0$), then $y \notin V_a$, because $y_i \ge q$ for $0 \le i \le N-1$, and $y_{N} \notin (a,1])$. Likewise, $\sigma^{m}(V_a) \cap V_a = \emptyset$ for $N+1 > m>0$. If $m\ge N+1$, then $\sigma^{m}(V_a) \cap V_a = \emptyset$ since any point $z=(z_0,z_1,\ldots)$ of $\sigma^{m}(V_a)$ has each coordinate less than or equal to $q <a$ and so cannot be in $V_a$. Hence, each such $x$ is wandering.

\vphantom{}

\item Suppose $x$ is eventually  fixed but is not eventually in $\mathbf{T}^{(p,q)}$. Then either there is some $z>q$ or some $z <p$ such that for some $m>0$, $\sigma^m(x)=\overline{z}$. If there is $z >q$ such that for some $m>0$, $\sigma^m(x)=\overline{z}$, then each coordinate of $x$ is greater than $q$, and $x$ is wandering (for the same reason as it was wandering in $\mathbf{G}$).

\vphantom{}

If there is $z <p$ such that for some $m>0$, $\sigma^m(x)=\overline{z}$, then  there is some greatest $N$ such that $x_N>z$. Let $x_N > a >z$. Then $x \in $ $U_a=\mathbf{W}^{(p,q)} \cap ( (a,1]^{N+1} \times [0,a) \times I^{\infty})$. Again, for each $m$, $U_a \cap \sigma^m(U_a) = \emptyset$, so $x$ is wandering.

\vphantom{}

\item Suppose $x$ is neither eventually  fixed nor eventually in $\mathbf{T}^{(p,q)}$. Then either each $x_i >q$, or eventually $x_i<p$. If each $x_i >q$, then there is least $n$ such that $x_n>x_{n+1}$ (since $x$ is not eventually fixed). Then there is $a$ such that $x_n > a > x_{n+1}$. If $U_a=\mathbf{W}^{(p,q)} \cap ((a,1]^{n+1} \times [0,a) \times I^{\infty}$, then $U_a$ is wandering and contains $x$, so again we have a wandering point.

\vphantom{}

Suppose that eventually $x_i <p$. Then there is least $n$ such that $x_n <p$ and $x_n>x_{n+1}$ (again, $x$ is not eventually fixed). Then if $x_n > a > x_{n+1}$, and $U_a=\mathbf{W}^{(p,q)} \cap ((a,1]^{n+1} \times [0,a) \times I^{\infty}$, then $U_a$ is wandering and contains $x$, so such a point is wandering.

\end{enumerate}

\vphantom{}

\begin{prop}
Let $\mathbf{W}^{(p,q)}:=\star_{i=1}^{\infty} (G \cup \{(p,q)\})$. Then $\mathbf{W}^{(p,q)}$ contains $\mathbf{T}^{(p,q)}$ and contains $\Delta$. Under $\sigma$, $\mathbf{T}^{(p,q)}$ is invariant and closed in $\mathbf{W}^{(p,q)}$, as is $\Delta$. Furthermore,  all points of $\mathbf{W}^{(p,q)} \setminus (\mathbf{T}^{(p,q)} \cup \Delta)$ wander.

\end{prop}

\begin{theorem} Suppose $D$ is a closed region in $[0,1] \times [0,1]$ with nonempty connected interior and such that if $D^{\circ}$  is the interior of $D$, then $D= \overline{D^{\circ}}$. Then either $\ent(D)=0$ or $\ent(D)=\infty$. 
\end{theorem}

\begin{proof} Either $D \subset G$, or $D \subset (\overline{[0,1] ^2 \setminus G})$, or $D^{\circ}$ intersects the diagonal $\{(x,x): 0 \le x \le 1 \}$. If $D \subset G$, or $D \subset (\overline{[0,1] ^2 \setminus G})$, then $\ent(D)=0$. Otherwise, $D^{\circ}$ intersects the diagonal, and there is some $0 \le p < q \le 1$ such that  $T^{(p,q)} \subset D$, so $\ent(D)=\infty$.

\end{proof}

In the following figure we give an example of such a set $D$ such that $D^{\circ}$ is not connected and the previous theorem doesn't hold. Hence, we need the assumption that the interior $D^{\circ}$ of $D$ is connected.

\vphantom{}

\begin{figure}[h!]
	\centering
\begin{tikzpicture}[scale=0.7]
\begin{axis}[
unit vector ratio*=1 1 1, 
xtick={0,0.5,1}, xticklabels={0,,1},
 ytick={0,0.5,1}, yticklabels={0,,1},
xmin=0,     
xmax=1,   
ymin=0,     
ymax=1,
] 

\node[mark size=4pt,color=blue!50!white] at (axis cs: 0.6,0.33) {\pgfuseplotmark{square*}};

\draw[mark=square*,mark size=22pt,mark options={color=blue!50!white}] plot coordinates {(axis cs: 0.62,0.35)};
\draw[mark=square*,mark size=22pt,mark options={color=blue!50!white}] plot coordinates {(axis cs: 0.35,0.62)};
\addplot [dashed,black] {x};
\end{axis}
\end{tikzpicture}
\caption{ }
\end{figure}

We are finally finished with the Triangle Example family. We turn not to an entirely different set of conditions, conditions often satisfied quite naturally, and consider the consequences for entropy.

\vphantom{}

Suppose $H$ is a closed subset of $[0,1]^{N+1}$ and $H$ contains $k$ mutually disjoint closed sets $H_1,\ldots,H_k$ such that $\pi_0(H_i)=[0,1]$ for each $1 \le i \le k$. Then if $s=s_1,s_2,\ldots$ is a sequence each member of which is in $\{1,\ldots,k\}$, $\mathbf{H}_s:=H_{s_1}\star H_{s_2} \star \cdots$ is a nonempty closed subset of $[0,1]^{\infty}$ and $\pi_0(\mathbf{H}_s)=[0,1]$. Furthermore, if $s$ and $t$ are different sequences of members of $\{1,\ldots,k\}$, then $\mathbf{H}_s \cap \mathbf{H}_t = \emptyset$.

\vphantom{}

To see that each $\mathbf{H}_s\neq \emptyset $ and $\pi_0(\mathbf{H}_s)=[0,1],$ let $x\in [0,1].$ Then there is $(x_0,\ldots,x_N) \in H_{s_1}$ such that $x_0=\pi_0 \left( (x_0,\ldots,x_N) \right)=x.$ But then $x_N\in [0,1] = \pi_0 (H_{s_2}),$ so there is some $(x_N,\ldots,x_{2N})\in H_{s_2}.$ We can continue this process, and find a point $\mathbf{x}=(x,x_1,\ldots,x_N,\ldots,x_{2N},\ldots)=(x_0,x_1,\ldots) \in \mathbf{H}_s.$ 
Furthermore if $s=(s_1,s_2\ldots), t=(t_1, t_2\ldots)$ are different sequences, each member of which is in $\{ 1,\ldots,k \},$ then there is least $i$ such that $s_i\neq t_i.$ Then $\mathbf{x}=(x_0,x_1,\ldots)\in \mathbf{H}_s$ means that $\pi_{ \langle (i-1)N,\ldots,iN \rangle}(\mathbf{x})=(x_{(i-1)N},\ldots,x_{iN}) \in H_{s_i}$ and $\pi_{ \langle (i-1)N,\ldots,iN \rangle}(\mathbf{x})=(x_{(i-1)N},\ldots,x_{iN}) \notin H_{t_i}$ since $H_{s_i} \cap H_{t_i}.$

\vphantom{}

 Let $S$ denote the set of all sequences each member of which is in $\{1,\ldots k\}$ and $\mathcal{H}=\{\mathbf{H}_s: s \in S\}$. ($S$ is a Cantor set and $\cup \mathcal{H}=\mathbf{H}$ is closed in $[0,1]^{\infty}$.) 

\vphantom{}

Then we have the following result:

\begin{theorem} The entropy of $H$ is at least $\log(k)$, i.e., $\ent(H) \ge \log(k)$.

\end{theorem}

\begin{proof} Since $\pi_0(H_i)=[0,1]$, $N(H_i, \alpha^{N+1}) \ge n$ for each $i$, and $N(H,\alpha^{N+1}) \ge kn$ for $\alpha$ with sufficiently small intervals. For each $m \ge 1$, there are $k^m$ sets $H_{s_1}\star H_{s_2} \star \cdots \star H_{s_m}$, with this collection being mutually disjoint. If $L_m=\{H_{s_1}\star H_{s_2} \star \cdots \star H_{s_m}: s_i \in \{1,..,k\} \text{ for each } 1 \le i \le m \}$, then $N(L_m,\alpha^{mN+1}) \ge k^m n$, and 

$$ \lim_{m \to \infty} \frac{\log(N(L_m, \alpha^{mN+1}))}{m}\ge \lim_{m \to \infty} \frac{\log(k^m n)}{m}= \lim_{m \to \infty} \frac{m\log(k)+\log(n)}{m}=\log(k).$$

\vphantom{}

Since $$\lim_{m \to \infty} \frac{\log(N(L_m, \alpha^{mN+1}))}{m} \le \lim_{m \to \infty} \frac{\log(N(H, \alpha^{mN+1}))}{m},$$ the result follows.

\end{proof}

\section{Entropy versus box counting dimension}

\vphantom{}

Superficially at least, defining entropy with box covers appears to be similar to the way the box counting dimension of a set is defined. For a minimal open cover $\alpha$ of $[0,1]$, its corresponding grid cover $\alpha^{n+1}$ of $\Pi_{i=0}^{n}I_{i}$, and a subset $A$ of  $\Pi_{i=0}^{n}I_{i}$, let $N^{*}(A, \alpha^{n+1})$ denote the number of members of $\alpha^{n+1}$ that intersect $A$. Let $|\alpha|$ denote the length of the largest interval in $\alpha$. The \textit{box counting dimension} of $A$ can be defined as $$\lim_{|\alpha| \to 0} \frac{\log(N^{*}(A, \alpha^{n+1})}{\log |\alpha|},$$ provided this limit exists. Unfortunately it does not always exist. 

\vphantom{}

For us, the closed sets $G$ in $I_{0} \times I_{1}$ in which we have been interested so far, have topological dimension the same as their box counting dimension. And this is also true of $\star_{i=1}^{n}G$, i.e., the box counting dimension of $\star_{i=1}^{n}G$ is the same as the topological dimension of $\star_{i=1}^{n}G$. Hence, relative to box counting dimension, our sets are rather simple.

\vphantom{}

Box counting dimension is intended to measure the dimension of a fractal set. There are many other ways to measure this dimension, such as Hausdorff dimension and correlation dimension. They are not equivalent in general. On the other hand, entropy is designed to measure how much points of a set interact with other points of the set (in a very particular way). 

\vphantom{} 

Consider the following example:

\begin{exam} Suppose $G=\{ (0,0),(0,1),(1,0),(1,1) \}$. Then the box counting dimension of $G$ is $0$, as is the box counting dimension of $\star_{i=1}^{n}G$ for each $n$, since $\star_{i=1}^{n}G$ is finite for each $n$. However, the entropy of $G$ is $\log2$,  since for each minimal open cover $\alpha$ of $[0,1]$ by open intervals, $$\frac{\log(N(\star_{i=1}^{n}G,\alpha^{n+1})}{n}=\frac{\log(2^{n+1})}{n}=\frac{(n+1)\log2}{n},$$ so $$\lim_{n \to \infty}\frac{\log(N(\star_{i=1}^{n}G,\alpha^{n+1})}{n}=\lim_{n \to \infty}\frac{(n+1)\log2}{n}=\log2.$$ It follows that $\sup_{\alpha}\ent(G,\alpha)=\log2$.

\end{exam}

Similar results hold for most, if not all, our examples. But we provide one last example that shows box counting dimension can also be greater than the topological entropy.

\begin{exam}
	Let $G=\{(x,x):0 \le x \le 1 \}$, i.e., $G$ is the diagonal from $(0,0)$ to $(1,1).$ in $[0,1]^2$. Then $\ent(G)=0$ and $\dim_{box} (\star_{i=1}^m G)=1$ (\cite{F}, p. 48).
\end{exam}

\end{document}